\documentclass{article}
\usepackage{graphicx,amsmath,amssymb,amsfonts,amsthm, comment,color,bm} 

\usepackage[algoruled]{algorithm2e}

\newtheorem{theorem}{Theorem}
\newtheorem{lemma}[theorem]{Lemma}
\newtheorem{corollary}[theorem]{Corollary}

\newtheorem{remark}{Remark}
\newtheorem{assumption}{Assumption}

\newcommand{\offdiag}{{\rm offdiag}}
\newcommand{\ones}{{\bf 1}}
\newcommand{\bv}{{\bm v}}
\newcommand{\bw}{{\bm w}}
\newcommand{\bx}{{\bm x}}

\newcommand{\bd}{{\bm d}}
\newcommand{\bt}{{\bm t}}
\newcommand{\bg}{{\bm g}}
\newcommand{\by}{{\bm y}}
\newcommand{\be}{{\bm e}}
\newcommand{\bz}{{\bm z}}

\newcommand{\bb}{{\bm b}}
\newcommand{\bp}{{\bm p}}
\newcommand{\bsu}{{\bm u}}
\newcommand{\wh}{\widehat}
\newcommand{\wt}{\widetilde}
\newcommand{\dotle}{\dot{\le}}
\newcommand{\fl}{\hbox{\rm fl}}
\newcommand{\sml}{\footnotesize}

\definecolor{mypink1}{rgb}{0.858, 0.188, 0.478}
\definecolor{mypink2}{RGB}{219, 48, 122}
\definecolor{mypink3}{cmyk}{0, 0.7808, 0.4429, 0.1412}
\definecolor{mygray}{gray}{0.7}
\definecolor{mygray1}{gray}{0.6}
\definecolor{myred}{rgb}{1,0,1}

\title{Component-wise accurate computation of the square root of an M-matrix}

\author{Dario A. Bini\thanks{Dipartimento di Matematica Universit\`a di Pisa} \and Bruno Iannazzo\thanks{Università degli Studi di Perugia}\and Beatrice Meini\thanks{Dipartimento di Matematica Universit\`a di Pisa} \and Jie Meng\thanks{Ocean University of China}}

\begin{document}
	\maketitle

\begin{abstract}
Component-wise accurate algorithms for computing the principal square root of an M-matrix are designed in terms of triplet representations.  A triplet representation of an M-matrix $A$ is the triple $(P,\bsu,\bv)$, where the matrix $P$ is such that $p_{ij}=-a_{ij}$ for $i\ne j$, $p_{ii}=0$, and $\bsu>0$, $\bv\ge 0$ are two vectors such that $A\bsu=\bv$. It is shown that if $A$ is an M-matrix representable by a triplet, then its principal square root exists and is an M-matrix represented by a triplet as well. New versions of the Cyclic Reduction and the Incremental Newton iterations are provided  in terms of triplets, to compute the principal matrix square root of $A$. It is shown that these algorithms are component-wise numerically stable independently of the singularity of $A$ and of its condition number. Numerical experiments are shown to confirm the component-wise stability.
\end{abstract}


\section{Introduction}\label{sec:intro}
An M-matrix is a real square matrix $A$ such that $A=\alpha I-B$, where $I$ is the identity matrix, $B$ is component-wise nonnegative (nonnegative for short), and $\alpha\ge \rho(B)$, the spectral radius of $B$.
This class of matrices arises in a wide range of applications, spanning from network theory to applied probability and the discretization of PDEs. This reason has made them attractive in linear algebra, where their properties have been extensively investigated \cite[Ch. 6]{bp:book}.
An interesting feature of M-matrices is that if $\alpha>\rho(B)$ then the matrix $A$ is invertible and the inverse is nonnegative. Thus, also the solution of the linear system $A\bx=\bb$, where $A$ is an invertible M-matrix and $\bb$ is nonnegative, is a nonnegative vector.

The customary theory of numerical conditioning gives bounds for the forward error, in norm, of the solution of a linear system subject to perturbations. If $\wt A$ and $\wt {\bm b}$ are sufficiently small perturbations of the invertible matrix $A$ and the nonzero vector $\bm b$, respectively, and if $\wt {\bx}$, $\bx$ are such that $A\bx=\bm b$ and $\wt A\wt {\bx}=\wt {\bm b}$, then it is well known that, for any operator norm $\|\cdot\|$, the following inequality holds
\[
    \frac{\|\wt {\bx}-\bx\|}{\|\bx\|}\le \frac{\mu(A)}{1-\mu(A)\varepsilon_A}(\varepsilon_A+\varepsilon_b),\qquad 
    \varepsilon_A = \frac{\|\wt A-A\|}{\|A\|},\qquad \varepsilon_b=\frac{\|\wt {\bm b}-{\bm b}\|}{\|{\bm b}\|},
\]
where $\mu(A)=\|A\|\|A^{-1}\|$ is the condition number of $A$. Thus, $\mu(A)$ can be viewed as an amplification factor of the error in the solution of linear systems.

Nevertheless, the positivity features of an M-matrix allow one to give, under mild conditions, uniform bounds on the component-wise forward relative error where the condition number magically disappears.
This property can be obtained through the {\em triplet representation} of the M-matrix $A$ \cite{poloni15,li12}, that consists in the triple $(P,\bsu,\bv)$, where $P=(p_{ij})_{i,j=1,\ldots,n}$ is the nonnegative matrix such that $p_{ij} = -a_{ij}$ for $i\ne j$ and $p_{ij}=0$ for $i=j$, while $\bsu$ and $\bv$ are a positive and a nonnegative vector, respectively, such that $A\bsu=\bv$. 

If $A$ is a nonsingular M-matrix with the triplet representation $(P,\bsu,\bv)$ and if $\wt A$ has a triplet representation $(\wt P,\bsu,\wt \bv)$ such that $|\wt p_{ij}-p_{ij}|\le \varepsilon| p_{ij}|$ for $i\ne j$ and $|\wt v_i-v_i|\le \varepsilon v_i$, then for a sufficiently small $\varepsilon$, the matrix $\wt A$ is a nonsingular M-matrix and one gets an upper bound,  independent of $\mu(A)$, for the relative variation of the entries of the inverse  (see \cite[Lemma 2.1]{alfa} and \cite[Thm. 2.2 and Eq. 2.7]{li12}). From this property, one deduces the following inequality for $\bx$ and $\wt \bx$ such that $A\bx =\bb$ and $\wt A\wt \bx=\wt \bb$, with nonnegative $\bb$ and $\wt{\bb}$:
\[
    {|\wt x_i-x_i|}\
    \le {|x_i|} 
    \gamma n \varepsilon+o(\varepsilon),
\]
where $\gamma$ is a small constant independent of $A$. Similar results can be obtained when there is a perturbation also on the vector $\bsu$.
This error bound is practically attained by the GTH-like algorithm \cite{gth}, which is a variant of Gaussian elimination based on the triplet representation \cite{poloni15}. A nice feature of the GTH algorithm, with respect to Gaussian elimination, is that there is no cancellation, that is, numbers of opposite sign are never added. 

The use of the triplet representation is not limited to the component-wise accurate solution of linear systems with nonnegative right-hand side, but can be used for the component-wise accurate solution of other basic quantities, such as the inverse $A^{-1}$,  the Schur complement of $A$ (that is known to be an M-matrix), a nonnegative eigenvector \cite{alfa}.
It has also been used to accurately compute the solution of some more complicated problems, such as the minimal nonnegative solution of algebraic Riccati equations  \cite{li21,li20,poloni15,li17} associated with M-matrices.
Other uses of the triplet representation include the solution of Sylvester equations \cite{li12} and the solution of quadratic matrix equations  \cite{li19,li22,poloniCR} such as those encountered in Markov Modulated Processes.

If the M-matrix $A$ is invertible or singular with semisimple zero eigenvalue, then there exists a unique solution of the matrix equation $X^2=A$ that is a polynomial of $A$ and whose nonzero eigenvalues have positive real part (see the next Lemma \ref{thm:1}). This solution, which turns out to be an M-matrix (see the next Theorem \ref{thm:triplsqrt}), is denoted by $A^{1/2}$ and is called the {\em principal square root} of the matrix $A$.

In this paper, we are interested in the accurate computation of the principal square root of an M-matrix assigned in terms of its triplet representation.
This problem finds some applications in the analysis of fractional diffusion problems in complex network analysis \cite{fabio}. Moreover, the square root and inverse square root of a matrix have many applications in a variety of scientific computing and engineering applications \cite{frommer}. For instance, the square root and inverse square root of discretized differential operators arise when computing Dirichlet-to-Neumann and Neumann-to-Dirichlet maps \cite{druskin}.

Here, we propose algorithms for computing the principal square root of an M-matrix $A$, based on the triplet representation, which are component-wise numerically stable. More precisely, given a triplet representation of $A$, we prove that $A^{1/2}$ has a triplet representation and we compute it with algorithms designed such that in the overall computation, numerical cancellation is never encountered.

Since the state-of-the-art algorithm for computing the principal square root, based on the Schur form of $A$ \cite{Deadman}, is not suited for the triplet representation, we will focus our attention on iterative methods based on the Newton iteration that produce sequences of M-matrices. 
The variants of the Newton method we consider are the Cyclic Reduction (CR) of \cite{meini}, and the Incremental Newton (IN) iteration of \cite{Iannazzo2003273}, see also \cite[Chapter 7]{higham:book}. For these algorithms, we provide suitable versions where at each step the triplet representing the current approximation is updated by performing only multiplications, divisions, and additions of positive numbers. This way, unlike the standard implementation, the sequence of approximations can be computed without numerical cancellation, ensuring that each step is component-wise accurate. Consequently, the overall computation yields a component-wise numerically stable approximation of the square root, represented in terms of triplets. This is confirmed by numerical experiments where our algorithms compute all the entries of $A^{1/2}$ with relative numerical accuracy close to the machine precision, while the MATLAB command {\tt sqrtm}, as well as the standard implementation of the customary algorithms, fail for the entries of smallest modulus.
It must also be observed that the implementation of these algorithms, given in terms of triplets, does not increase the computational complexity of each step or  affect the convergence speed of the sequences generated by CR and IN.

The paper is organized as follows. In Section \ref{sec:prel}, we recall some useful results on M-matrices and some properties of the triplet representation of an M-matrix.
Section \ref{sec:sqrt} summarizes the properties of the square roots of M-matrices.
In Section \ref{sec:crin} we deal with algorithmic issues. In particular, in Section \ref{sec:cr}, we reformulate in terms of triplets the algorithms of \cite{meini} based on Cyclic Reduction together with its accelerated version, and in Section \ref{sec:in}, the algorithm of \cite{Iannazzo2003273} based on the Incremental Newton.
Section \ref{sec:numan} outlines some results useful to perform the rounding error analysis of a single step of the methods introduced in the previous sections. Section \ref{sec:numexp} presents the results of some numerical experiments, which show that the new proposed algorithms maintain a component-wise relative error bounded by a small multiple of the machine precision, whereas the classical versions produce a very large component-wise relative error.
Finally, Section \ref{sec:conc} contains final remarks and conclusions.

\section{Preliminaries}\label{sec:prel}
\subsection{M-matrices}
We recall the main definition and properties of M-matrices. We refer the reader to the book \cite{bp:book}, for more details.

A real matrix $A=(a_{ij})_{ij}$ is said to be nonnegative if $a_{ij}\ge 0$ for any $i,j$, and we write $A\ge 0$, is positive if $a_{ij}>0$ for any $i,j$, and we write $A>0$.
An $n\times n$ real matrix  $A$ is said to be an  M-matrix if it can be written as
$A=\alpha I-B $,  where $B\ge 0$, $\alpha\ge \rho(B)$, and $\rho(B)$ is the spectral radius of $B$.
An M-matrix $A$ has ``Property C'' if it can be written as $A=\alpha I -B$ with $B$ nonnegative and the eigenvalues of $B/\alpha$ belong to $\{z\in\mathbb C\,:\,|z|<1\}\cup\{1\}$, where if $1$ is an eigenvalue of $B$ then it is semisimple. Property C is equivalent to the fact that $A$ is nonsingular or is singular with semisimple $0$ eigenvalue \cite[Thm. 4.11]{bp:book}.

The following properties of M-matrices will be used in the sequel, we refer the reader to \cite[Chapter 6]{bp:book} for more details.

\begin{theorem}\label{1}
Let $A$ be an M-matrix. Then
\begin{itemize}
\item[(1)] The principal submatrices are M-matrices;
\item[(2)] The nonzero eigenvalues of $A$ lie in the open right half complex plane;
\item[(3)] $\Pi^\top  A\Pi$ is an M-matrix if $\Pi$ is a permutation matrix.
\end{itemize}
\end{theorem}

\begin{theorem}\label{2}
Let $A$ be a nonsingular M-matrix. Then
\begin{itemize}
\item[(1)] There exists a vector $\bsu>0$ such that $A\bsu\ge 0$;
\item[(2)] $A^{-1}\ge 0$ and $A^{-1}>0$ if $A$ is irreducible.
\end{itemize}
\end{theorem}

\begin{theorem}\label{3}
Let $A\in\mathbb R^{n\times n}$ be a singular irreducible M-matrix. Then
\begin{itemize}
\item[(1)] There exists a vector $\bsu>0$ such that $A\bsu=0$;
\item[(2)] $A$ has rank $n-1$;
\item[(3)] Each principal submatrix of $A$ other than $A$ is a nonsingular M-matrix;
\item[(4)] If $A\bx\ge 0$ then $A\bx=0$;
\item[(5)] $A$ has Property C.\end{itemize}
\end{theorem}

\subsection{Triplet representations of an M-matrix}
\label{sec:trip}

A {\em triplet representation} of the M-matrix $A$, is a triple $(P,\bsu,\bv)$, where $P\in\mathbb R^{n\times n}$ is the nonnegative matrix such that $p_{ij} = -a_{ij}$ for $i\ne j$ and $p_{ij}=0$ for $i=j$, while $\bsu\in\mathbb{R}^n$ and $\bv\in\mathbb{R}^n$  are such that $A\bsu=\bv$, and $\bsu>0$, $\bv\ge 0$. 

For the triplet representation, we will use the notation $(-\offdiag(A),\bsu,\bv$), where $-\offdiag(A)$ is the matrix that coincides with $-A$ in the off-diagonal entries and has null entries on the diagonal.

If  $A$ is invertible or singular irreducible, then by Theorem \ref{2} and \ref{3} a triplet 
representation $(P,\bsu,\bv)$ exists. If  $A$ is singular and irreducible, then a triplet representation is such that $\bv=0$.

The triplet representation may not be unique, and the entries of the matrix $A$ can be easily obtained from any triplet representation, without subtractions of nonnegative numbers, by means of the vector
$\by=P\bsu$ and the
simple formula
\begin{equation}\label{eq:diag}
a_{ij}=\left\{\begin{array}{ll} -p_{ij}, &\hbox{if }i\ne j,\\[1ex]
\displaystyle{\frac{v_i+y_i}{u_i}},& \hbox{if }i= j.
\end{array}\right.
\end{equation}

In some applications, the triplet is available or can be computed exactly. For instance, for the Laplacian $A$ of a network, we know that $A{\ones} =0$, where $\ones=[\begin{array}{cccc}1&1&\ldots&1\end{array}]^\top$, or, for the discretization of the Laplace operator on a segment with central differences and zero boundary conditions, we have $A{\bf 1} = \begin{bmatrix} 1 & 0 & \cdots & 0 & 1\end{bmatrix}^\top $. 

A singular reducible M-matrix can have no triplet representation. An example is
$A=\left[  \begin{smallmatrix}
    0 & -1\\0 & 0
    \end{smallmatrix}\right].
$
Note that in this case $0$ is not semisimple, and indeed, we will show that a necessary condition for a singular matrix $A$ to possess a triplet representation is that 0 is a semisimple eigenvalue.

Apparently, the existence of a triplet cannot be related to a spectral property; indeed, the two singular reducible M-matrices
\begin{equation}\label{eq:A1A2}
    A_1=\begin{bmatrix}
    1 & -1\\0 & 0
    \end{bmatrix},\qquad
    A_2=\begin{bmatrix}
    1 & 0\\ -1 & 0
    \end{bmatrix},
\end{equation}
have the same spectral properties, in particular $0$ is simple, but, while $A_1$ has the triplet representation
$
    P = \left[\begin{smallmatrix} 
    0 & 1\\ 0 & 0
    \end{smallmatrix}\right],~
    \bsu = \left[\begin{smallmatrix} 1\\1\end{smallmatrix}\right],~
    \bv = \left[\begin{smallmatrix} 0\\0\end{smallmatrix}\right]
$,
the matrix $A_2$ has no triplet representation.

A necessary and sufficient condition for the existence of the triplet can be related to the Frobenius normal form of $A$, that is a matrix of the form
\begin{equation}\label{eq:1}
    B=\Pi^\top  A \Pi = \begin{bmatrix}
    B_{11} & B_{12} & \cdots & B_{1\ell}\\
    & B_{22} & \cdots & B_{2\ell}\\
    & & \ddots & \vdots\\
    & & & B_{\ell\ell}\end{bmatrix},
\end{equation}
where $\Pi$ is a permutation matrix and the diagonal blocks are irreducible matrices. A diagonal block can be  $1\times 1$ zero block that for us is an irreducible matrix.

The following result can be deduced from \cite[Thm. 4]{01} and can be found also in \cite[Thm. 1-(ii)]{02}.

\begin{theorem}\label{thm:exist}
Let $A\in\mathbb R^{n\times n}$ be an M-matrix. The matrix $A$ has a triplet representation if and only if its Frobenius normal form \eqref{eq:1}
is such that if a diagonal block, say $B_{ii}$, is singular, then the off-diagonal blocks $B_{ij}$ for $j\ne i$, are zero.
\end{theorem}

Theorem \ref{thm:exist} and the fact that by Theorem \ref{1} the diagonal submatrices of an M-matrix are M-matrices as well imply immediately the following result (whose first part has also been proved in \cite[Lemma 4]{guo2}).

\begin{corollary}\label{cor:1}
If $A\in\mathbb R^{n\times n}$ is an M-matrix with triplet representation, then $A$ is either nonsingular, or $0$ is a semisimple eigenvalue.
Moreover, there exists a permutation matrix $\Pi$ such that
\begin{equation}\label{eq:2}
\Pi^\top A\Pi = \begin{bmatrix}
A_{11} & A_{12}\\
0 & A_{22}
\end{bmatrix},\qquad
A_{22} = \begin{bmatrix}
A_{22}^{(1)}\\
& \ddots \\
& & A_{22}^{(\nu)}
\end{bmatrix},
\end{equation}
where $A_{11}$ is a nonsingular (or empty) M-matrix and $A_{22}$ is block diagonal with irreducibile singular diagonal blocks $A_{22}^{(1)},\ldots,A_{22}^{(\nu)}$, that are M-matrices.
\end{corollary}

Corollary \ref{cor:1} provides only a necessary condition. Indeed, the set of M-matrices with a triplet is a strict subset of M-matrices with  Property C (nonsingular or with a semisimple $0$ eigenvalue). For instance, the matrix $A_2$ in \eqref{eq:A1A2}, lacking a triplet representation, has  Property C, because its eigenvalues are $0$ (semisimple) and $1$.

In summary, we get the following strict inclusions
\[
    \mathcal M_1\subset \mathcal M_2\subset \mathcal M_3\subset \mathcal M_4,
\]
where $\mathcal M_4$ is the set of  M-matrices, $\mathcal M_3$ is the set of M-matrices with  Property~C, $\mathcal M_2$ is the set of M-matrices with triplet and $\mathcal M_1$ is the set of nonsingular or singular irreducible M-matrices.
The matrices in $\mathcal M_2$ are also called regular in \cite{guo1,guo2} and matrices with property $\mathcal W\mathcal S$ in \cite{02}.

The LU factorization of an M-matrix with triplet representation can be computed by means of a component-wise numerically stable algorithm, based on the extension of the GTH algorithm \cite{alfa}, as in the version described in \cite{poloni15}.
If $A$ is a nonsingular M-matrix and $B\in\mathbb{R}^{n\times m}$, $B\ge 0$, then   also the solution of the system
$AX=B$ can be computed in a component-wise numerically stable way, relying on the $LU$ factorization of $A$. The overall procedure is reported in Algorithm~\ref{alg:gth}. Note that the number of arithmetic operations (ops) needed for computing the LU factorization is $\frac{2}{3}n^3+o(n^3)$, that is asymptotically equivalent to the number of ops needed by Gaussian elimination.

\begin{algorithm}
\sml
\textbf{Input}:
 $B\in\mathbb R^{n\times m}$, $B\ge 0$, and a triplet representation $(-\offdiag(A), \bv, \bw)$ of the matrix $A\in\mathbb R^{n\times n}$.

\textbf{Output}: $X = A^{-1} B$.

Set $L=I$, $U=O$, $\offdiag(U)=\offdiag(A)$\;

\For{$\ell=1,2,\dots,n-1$}{

$U_{\ell,\ell} = (w_\ell - U_{\ell,\ell+1:n} v_{\ell+1:n} )/v_\ell$\;

$L_{\ell+1:n,\ell} = U_{\ell+1:n,\ell} /U_{\ell,\ell}$\;

$w_{\ell+1:n} = w_{\ell+1:n} - L_{\ell+1:n,\ell} w_\ell$\;

$\offdiag(U_{\ell+1:n,\ell+1:n}) = \offdiag(U_{\ell+1:n,\ell+1:n} - L_{\ell+1:n,\ell} U_{\ell,\ell+1:n} )$\;
}
$U_{n,n}=w_n/v_n$\;
Compute $Y = L^{-1} B$ by forward substitution\;
Compute $X = U^{-1}Y$ by back substitution.

\caption{\sml GTH-like algorithm for solving the system $AX = B$, where $A\in\mathbb R^{n\times n}$ is a nonsingular M-matrix,  $X,B\in\mathbb R^{n\times m}$, and $B\ge 0$} \label{alg:gth}
\end{algorithm}

\section{Properties of the square root of an M-matrix}
\label{sec:sqrt}

The square roots of the matrix $A$ are the solutions of the equation $X^2=A$. There can be finitely many, infinitely many or no square roots of a matrix $A$. These are usually distinguished in primary square roots that can be written as a polynomial in $A$ and nonprimary square roots that cannot.

We recall and synthesize some results on the analysis of the existence of square roots of M-matrices. A classic reference is \cite{as82} where the existence is related to Property C (see Section \ref{sec:prel}).

The following lemma recalls some known properties.

\begin{lemma}\label{thm:1}
If $A$ has no negative real eigenvalues and is nonsingular or singular with semisimple $0$ eigenvalue, then $A$ has only one primary square root with nonzero eigenvalues in the open right half plane, that is the principal square root, denoted by $A^{1/2}$. If $A$ has a defective $0$ eigenvalue, then there exists no primary square root. 
\end{lemma}
\begin{proof}
Let $\lambda_1,\ldots,\lambda_n$ be the eigenvalues of $A$. The eigenvalues of a solution $X$ are square roots of the eigenvalues of $A$, and, since $A$ has no negative real eigenvalues, for any $\lambda_i\ne 0$, there is a unique square root in the open right half plane, namely $\lambda_i^{1/2}$.
The problem is thus reduced to the existence and uniqueness of a primary square root with eigenvalues $\lambda_1^{1/2},\ldots,\lambda_n^{1/2}$, that follows from \cite[Thm. 6.4]{u}, where the condition holds for a matrix with semisimple $0$ eigenvalues, with $f(x)=x^2$.
For the second part, we show that if there exists a primary square root, then 0 cannot be a defective eigenvalue.
If $X$ is a primary square root, then there exists a polynomial $p(x)$ such that $X=p(A)$ and $X^2=p(A)^2=A$. This implies that the minimal polynomial $\mu(z)$ of $A$ divides $p(z)^2-z$. If $X$ is singular then $z|\mu(z)$ and thus $z|p(z)$ that implies $z^2\not|\mu(z)$. This proves that a primary square root can exist only if $0$ is semisimple.
\end{proof}

In the sequel, the primary square root 
is denoted by $A^{1/2}$. In the following theorem, we recall the results that clarify the connections between $A^{1/2}$ and the M-matrices that solve the equation $X^2=A$.  

\begin{theorem}\label{th:abc}
Let $A$ be an M-matrix.
\begin{itemize}
\item[(a)] If $A$ is nonsingular or $0$ is a semisimple eigenvalue, then $X=A^{1/2}$ is the unique M-matrix that is a primary square root of $A$. If $A$ is nonsingular or $0$ is a simple eigenvalue, then $X=A^{1/2}$ is the unique M-matrix square root of $A$.
\item[(b)] If $A$ is singular with defective $0$ eigenvalue then there is no M-matrix square root of $A$.
\end{itemize}
\end{theorem}
\begin{proof}
(a) In the stated hypotheses, by Lemma \ref{thm:1}, the primary square root $A^{1/2}$ exists and is the unique primary square root with nonzero eigenvalues in the  open right half plane, then no other primary square root can be an M-matrix (since by Theorem \ref{1} the nonzero eigenvalues of an M-matrix belong to the right half plane). The fact that $A^{1/2}$ is an M-matrix concludes the proof.
If $A$ is nonsingular or $0$ is a simple eigenvalue, then it has property C, and by \cite{as82} we know that there exists a unique M-matrix square root, which must be $A^{1/2}$.

(b) This follows from \cite[Thm. 4]{as82}, since a matrix with defective $0$ eigenvalue has no Property C and thus it has no M-matrix square root.
\end{proof}

In case (a) of Theorem \ref{th:abc}, there may exist other M-matrices that are nonprimary square roots of $A$.
For instance, for the $0$ matrix, there exists just one M-matrix that is a primary matrix root, namely the $0$ matrix; while there are infinitely many M-matrices that are nonprimary matrix square roots, namely
$
  \left[\begin{smallmatrix} 0 & -a\\0 & 0\end{smallmatrix}\right],~
  \left[\begin{smallmatrix} 0 & 0\\-a & 0\end{smallmatrix}\right]~ a>0.
$

Corollary~\ref{cor:1} shows that if $A\in\mathbb R^{n\times n}$ is an M-matrix with triplet, then the matrix satisfies  Property C and thus by Lemma \ref{thm:1} there exists a principal square root. We prove that this matrix has a triplet representation. For an invertible matrix $A$ with triplet $(P,\bsu,\bv)$, there is an easy proof, since $A^{1/2}\bsu = A^{-1/2}A\bsu=A^{-1/2}\bv\ge 0$, so that $A^{1/2}$ has the triplets $(-\offdiag(A^{1/2}),\bsu,A^{-1/2}\bv)$ and $(-\offdiag(A^{1/2}),A^{1/2}\bsu,\bv)$. While for the case in which $A$ is not invertible, we need the following properties of primary matrix functions.
\begin{lemma}\label{lemma}
Let $A\in\mathbb C^{n\times n}$ and let $f$ be a function defined at the spectrum of $A$ and analytic at the defective eigenvalues of $A$.
\begin{itemize}
\item[(a)] For $M\in\mathbb C^{n\times n}$ invertible, we have $f(M^{-1}AM)=M^{-1}f(A)M$.
\item[(b)] Let $A=[A_{ij}]_{i,j=1,\ldots,\nu}$ be block upper triangular ($A_{ij}=0$ for $i>j$) [resp. block diagonal ($A_{ij}=0$ for $i\ne j$)]. Let $f(A)=[F_{ij}]_{i,j=1,\ldots,\nu}$ be the block partitioning of $f(A)$ performed with the same block partition as $A$. Then $f(A)$ is block upper triangular [resp. block diagonal] and $F_{ii}=f(A_{ii})$ for $i=1,\ldots,\nu$.
\item[(c)] If $A$ is reducible then $f(A)$ is reducible.
\end{itemize}
\end{lemma}
\begin{proof}
(a) is \cite[Thm. 1.13-c]{higham:book} and (b) is \cite[Thm. 1.13-f]{higham:book}.

(c) If $T:=\Pi^\top A\Pi$ is block upper triangular with nontrivial blocking for a permutation matrix $\Pi$, by (b) the matrix $f(T)$ is block upper triangular with nontrivial blocking and since by (a) one obtains $\Pi^\top f(A)\Pi=f(\Pi^\top A\Pi)=f(T)$, we deduce that $f(A)$ is reducible.
\end{proof}

Now we are ready to prove the main result of this section.

\begin{theorem}\label{thm:triplsqrt}
Let $A\in\mathbb R^{n\times n}$ be an M-matrix with triplet representation. Then there exists the principal square root of $A$, and it is an M-matrix with triplet representation.
\end{theorem}
\begin{proof}
By Corollary \ref{cor:1}, the matrix is nonsingular or singular with semisimple $0$  eigenvalue, so by Lemma \ref{thm:1} the square root $A^{1/2}$ exists and it is an M-matrix. Moreover, there exists a permutation matrix $\Pi$ such that $B:=\Pi^\top A\Pi$ has the form \eqref{eq:2}.
By Lemma \ref{lemma}-(a-b), with $f(z)=z^{1/2}$, we discover that $A^{1/2}=\Pi B^{1/2}\Pi^\top $ and
\[
B^{1/2} = \begin{bmatrix} A_{11}^{1/2} & \ast & \ast & \ast\\
& (A_{22}^{(1)})^{1/2} \\
& & \ddots \\
& & & (A_{22}^{(\nu)})^{1/2}
\end{bmatrix},
\]
where the diagonal blocks are of the same type (invertible, irreducible singular) as the diagonal blocks of $B$. This is immediate for invertible or $1\times 1$ zero blocks, while for nonzero irreducible singular blocks, we can use Lemma \ref{lemma}-(c).
The necessary and sufficient conditions for the existence of a triplet representation of the M-matrix $B^{1/2}$ is fulfilled. Finally, if $(-\offdiag(B^{1/2}),\bsu,\bv)$ is a triplet for $B^{1/2}$, then $(-\offdiag(A^{1/2}),\Pi\bsu,\Pi\bv)$ is a triplet for $A^{1/2}$.
\end{proof}

Using the proof of Theorem \ref{thm:triplsqrt}, one may prove that if $A$ is a matrix with no triplet representation, then no square root $X$ of $A$ has triplet representation, because the triplet representation imposes a structure on $X$ that is preserved by squaring.

Another interesting remark is that a matrix with a triplet may have M-matrix square roots without a triplet, e.g., the zero matrix that has $\left[\begin{smallmatrix} 0 & -a\\0 & 0 \end{smallmatrix}\right]$, with $a>0$ as M-matrix square roots without a triplet.

Finally, we demonstrate that our problem can be reduced to computing the square root of a matrix $A$ with a triplet that satisfies the following. 
\begin{assumption}\label{ass:2}
The matrix $A$ can be represented as $A=I-C$, where $C\ge 0$ and $\rho(C)\le 1$.
\end{assumption}

Let $A=(a_{ij})$ be an M-matrix assigned in terms of the triplet $(P,\bsu,\bv)$.
Without loss of generality, we may divide the matrix $A$ by a constant $\alpha>0$, replace $A$ with $\wh A=A/\alpha$, and recover $A^{1/2}=\alpha^{1/2} {\widehat A}^{1/2}$. For $\alpha$ sufficiently large, for instance $\alpha\ge \max a_{ii}$ (with $A\ne 0$), we may write $\wh A$ as $\wh A=I-C$, where $C=I- A/\alpha \ge 0$. 

Since we know that $\wh A$ is an M-matrix, then necessarily, $\rho(C)\le 1$. In fact, if $\rho:=\rho(C)>1$, then for the Perron-Frobenius theorem, there would exist $\bw\ge 0$ such that $C\bw=\rho\bw$, so that $\wh A\bw=(1-\rho)\bw$ where $1-\rho<0$. This fact contradicts the property that the real eigenvalues of an M-matrix are nonnegative (see Theorem \ref{1}).
Therefore, without loss of generality, we may restrict our attention to matrices satisfying Assumption \ref{ass:2}.

\section{A triplet version for Cyclic Reduction and Incremental Newton}
\label{sec:crin}

The Cyclic Reduction (CR) algorithm introduced in \cite{meini} and the Incremental Newton (IN) method proposed in \cite{Iannazzo2003273} are two iterations that produce (a multiple of) the same sequence obtained by applying the Newton method to the equation $X^2-A=0$ with $X_0=A$. When applied to an M-matrix, they produce a  sequence of M-matrices. When $A$ has a triplet representation, we show that all the iterates own a triplet representation and we provide an iteration based on the triplets, thus removing possible cancellation in the overall computation.

Throughout this section we assume that $A=(a_{ij})$ satisfies Assumption~\ref{ass:2}. This is obtained by exploiting the identity $A^{1/2}=(A/s)^{1/2}s^{1/2}$, with $s\ge \max_i a_{ii}$. Note that if $A$ has triplet representation $(-\offdiag(A),\bsu,\bv)$ then $A/s$ has representation $(-\offdiag(A)/s,\bsu,\bv/s)$.

\subsection{Cyclic Reduction}\label{sec:cr}

The CR algorithm applied to the palindromic matrix polynomial
$P(\lambda)=(A-I)\lambda^2 + 2(A+I)\lambda + A-I=0$, produces the sequences
\begin{equation}\label{cr}
    \begin{array}{ll}
       W_{\ell+1}=-W_\ell Z_\ell^{-1}W_\ell,& W_0=A-I,\\[1ex]
       Z_{\ell+1}=Z_\ell+2 W_{\ell+1}, &  Z_0=2(A+I),
    \end{array}
\end{equation}
$\ell=0,1,\ldots$, that are well defined when $A^{1/2}$ exists and $\lim_{\ell\to\infty} Z_\ell = 4A^{1/2}$ \cite{meini}. 
Under Assumption \ref{ass:2}, $W_0\le 0$ and the matrices $Z_\ell$, defined in \eqref{cr}, are nonsingular M-matrices, and $W_\ell\le 0$, $\ell=1,2,\ldots$

Concerning convergence,  if $A$ is nonsingular, then for any matrix norm,
\begin{equation}\label{eq:convspeedCR}
\|W_\ell\|=O(\rho(H)^{2^\ell}), ~~\|4A^{1/2}-Z_\ell\|=O(\rho(H)^{2^{\ell+1}}),
\end{equation}
where $H=(A^{1/2}-I)(A^{1/2}+I)^{-1}$ and $\rho(H)<1$, that is, convergence is quadratic. On the other hand, if $A$ is singular with simple $0$ eigenvalue, then $\|W_\ell\|=O(\frac{1}{2^\ell})$, and $\|4A^{1/2}-Z_\ell\|=O(\frac{1}{2^\ell})$, that is, convergence is linear with rate factor $1/2$.

If the matrix $A$ has triplet representation $(-\offdiag(A),\bsu,\bv)$, we prove that $Z_\ell$ has a triplet representation that can be computed in a stable way. This is obtained by the following.
\begin{theorem}\label{th:cr}
Let $A$ be an M-matrix fulfilling Assumption \ref{ass:2} and with triplet representation $(-\offdiag(A),\bsu,\bv)$. The sequence \eqref{cr} is well defined, moreover $(-\offdiag(Z_\ell),\bsu,\bv_\ell)$ is a triplet for $Z_\ell$, where $ \bv_\ell=\bp_\ell-2W_\ell\bsu$ 
with \begin{equation}\label{eq:trip}\begin{split}
&\bp_0=4\bv,\\
&\bp_{\ell+1}=\bp_\ell-2W_\ell Z_\ell^{-1}\bp_\ell,~~\ell=0,1,\ldots.
\end{split}
\end{equation}
\end{theorem}
\begin{proof}
Since $A$ has a triplet representation, then $A^{1/2}$ exists and \eqref{cr} is well defined. From $W_\ell\le 0$ and $Z_\ell^{-1}\ge 0$ we deduce that $\bp_\ell\ge 0$ and $\bv_\ell\ge 0$. It remains to show that $\bv_\ell=Z_\ell\bsu$, which is proved by induction. For $\ell=0$,
we have 
\[
    \bv_0=\bp_0-2W_0\bsu=4\bv-2(A-I)\bsu=2(\bv+\bsu)=2(A+I)\bsu=Z_0\bsu.
\]
For the inductive step, we get
\[
\begin{split}
    \bv_{\ell+1} & = \bp_{\ell+1} - 2W_{\ell+1}\bsu = (I-2W_\ell Z_\ell^{-1})\bp_\ell - 2W_{\ell+1}\bsu
    \\ & = (I-2W_\ell Z_\ell^{-1})(Z_\ell+2W_\ell)\bsu - 2W_{\ell+1}\bsu
    \\ & = Z_\ell\bsu - 4W_\ell Z_\ell^{-1}W_\ell\bsu  - 2W_{\ell+1}\bsu = (Z_\ell + 2W_{\ell+1})\bsu = Z_{\ell+1}\bsu,
\end{split}
\]
where we used \eqref{cr}, \eqref{eq:trip}, and $\bp_\ell=\bv_\ell+2W_\ell\bsu =(Z_\ell+2W_\ell)\bsu$.
\end{proof}

Theorem \ref{th:cr} leads to an algorithm for the square root without numerical cancellation except in the initialization $W_0=A-I$, provided that $A$ fulfills Assumption \ref{ass:2}. At step $\ell$, since $W_\ell\le 0$ and $\bp_\ell\ge 0$, we can compute $Z_\ell^{-1}W_\ell$ and  $Z_\ell^{-1}\bp_\ell$ accurately, via the triplet of $Z_\ell$. Then we can compute $\bp_{\ell+1}$, $\bv_{\ell+1}$, $W_{\ell+1}$ and $\offdiag(Z_{\ell+1})$ summing only positive numbers. 

The overall implementation of this procedure is given in Algorithm \ref{alg:cr}, where the symbol $\oslash$ stands for component-wise division.

\begin{algorithm}
\sml
\textbf{Input:} A triplet representation $(-\offdiag(A),\bsu,\bv)$ of the M-matrix $A$, where $\bsu>0$, $\bv=A\bsu\ge 0$.

\textbf{Output:} The entries of $Z=A^{1/2}$ and its triplet representation $(-\offdiag(Z),\bsu,\bw)$.

Set $s=\gamma\max_i a_{ii}$, (default: $\gamma=4$)\;
$A\gets A/s$, $\bv\gets \bv/s$\;

Set $W_0=A-I$, $Z_0=2(I+A)$, $\bp_0=4\bv$, $\bv_0=2(\bsu+\bv)$\;

\For{$\ell=0,1,2,\ldots$}{
     given the triplet $(-\offdiag(Z_\ell),\bsu,\bv_\ell)$ compute $[G_\ell,\bt_\ell]=Z_\ell^{-1}[W_k,\bp_\ell]$ by means of  Algorithm \ref{alg:gth}\;
     $W_{\ell+1}=-W_\ell G_\ell$\;
     $Z_{\ell+1}=Z_\ell+2W_{\ell+1}$\;
     $\bp_{\ell+1}=\bp_\ell-2W_\ell\bt_\ell$\;
     $\bv_{\ell+1}=\bp_\ell-2W_\ell\bsu$ \;
     \If{$\|W_{\ell+1}\|\le\varepsilon$}{exit\;}
    }
    $\bw=\bv_{\ell+1}$,\quad
    $\bd=(\bw-\offdiag(Z_{\ell+1})\bsu)\oslash\bsu$\;
    $Z=\frac{1}{4}(\hbox{diag}(\bd)+\offdiag(Z_{\ell+1}))$,\quad $\bw=\frac{1}{4}\bw$\;
$\offdiag(Z)\gets \offdiag(Z)s^{1/2}$,\quad $\bw\gets \bw s^{1/2}$,\quad $Z\gets Zs^{1/2}$.
\caption{\sml Computing $A^{1/2}$ by means of CR}\label{alg:cr} 
\end{algorithm}

In the initialization of Algorithm \ref{alg:cr}, one needs to set $s$ such that Assumption~\ref{ass:2} is fulfilled. We proved that for a matrix $\wh A=(\wh a_{ij})$, $s\ge \max_i \wh a_{ii}$ does the job with $A:=\wh A/s$.
From the point of view of the convergence it is better to choose $s$ not too large, since, for instance, if $A$ is nonsingular and $\|A\|$ is small, then $H$ is close to $-I$ an $\rho(H)$ is close to $1$ so that convergence slows down according to \eqref{eq:convspeedCR}.

On the other hand, in the initialization of Algorithm \ref{alg:cr}, the matrix $W_0=A-I$ must be computed, and the diagonal entries of $W_0$ are obtained through subtraction of nonnegative quantities. If $s$ is as small as possible, say $s=\max_i\wh a_{ii}$, this exposes the algorithm to numerical cancellation in case $\wh A$ is affected by an error. However, this issue can be mitigated by choosing $s$ sufficiently large so that the differences $(W_0)_{ii}=(\wh A)_{ii}/s-1$ do not much amplify the error present in $\wh A$. For instance, choosing $s=4\max_i a_{ii}$ yields an amplification of the error in the entries of $A$ at most by the factor $4/3$.

\begin{remark}\rm
If the matrix $A$ is such that $A\ones = 0$, as in the graph Laplacian, then $\bsu=\ones$ and $\bv=0$. Consequently, $\bp_\ell=0$ and $\bv_\ell=-2W_\ell\ones$ for any $\ell$. This fact simplifies the implementation of the algorithm. A similar property holds if $A\bsu=0$ for some vector $\bsu>0$, or if $A\bsu=\lambda \bsu$ since, in this case, $\bsu$ is an eigenvector of $W_\ell$ and $Z_\ell$.
\end{remark}

\begin{remark}\rm
If the matrix $A$ does not possess a triplet representation, while $A^{1/2}$ exists, the sequence \eqref{cr} is well defined with $Z_\ell$ invertible, so that $Z_\ell$ has a triplet representation, while the limit may not, e.g., $A=A^{1/2}=\left[\begin{smallmatrix} 1 & 0 \\ -1 & 0\end{smallmatrix}\right]$.
\end{remark}

\subsubsection{Accelerating convergence}\label{sec:shift}
When $A$ is singular with a semisimple $0$ eigenvalue,  the convergence of Cyclic Reduction is linear instead of quadratic. 
In this case, the shift technique, described in \cite[Section 3.6]{blm:book}, can be used to accelerate the convergence of CR for computing the square root under suitable conditions.

Assume that $A$ is a singular irreducible M-matrix so that, by Theorem \ref{3}, there exists a vector $\bsu>0$ such that $A\bsu=0$. Let $\bw\ge 0$ be such that $\bsu^\top \bw=1$, and set $Q=\sigma \bsu\bw^\top $, where $\sigma\in\mathbb R\setminus\{2\}$. Consider the following sequences generated by CR with rank-one modifications of the initial conditions in the previous section:
\begin{equation}\label{eq:crshifted}
\begin{array}{ll}
\wh W_{\ell+1}=-\wh W_\ell\wh Z_\ell^{-1}\wh W_\ell,&\wh W_0
=A-I+Q, \\
\wh Z_{\ell+1}=\wh Z_\ell+2\wh W_{\ell+1},&\wh Z_0=2(A+I)-Q.
\end{array}
\end{equation}

\begin{theorem}
Let $A$ be a singular irreducible M-matrix and $\{Z_\ell\}_\ell,\{W_\ell\}_\ell$ be the sequences obtained by \eqref{eq:crshifted}.
Then the matrix $\wh Z_0^{-1}\wh W_0$ has the same eigenvalues as $Z_0^{-1}W_0$ except for the eigenvalue $-1/2$ that is replaced by $\frac{\sigma-1}{2-\sigma}$.
In particular, if $\sigma\in (0,\frac{4}{3})$, then $\wh Z_0^{-1}\wh W_0$
has no real eigenvalues of modulus greater
than $1/2$, and
    $\lim\sup_\ell \| \wh W_\ell\|^\frac{1}{2^\ell}\le \xi$, $\lim\sup_\ell \|\wh Z_\ell-\wh S\|^\frac{1}{2^\ell}\le\xi^2$, with $0<\xi<1$ where
    $\wh S=\wh Z_0(I-4(\wh Z_0^{-1}\wh W_0)^2)^{1/2}$. 
\end{theorem}

\begin{proof}
The eigenvalues of 
$Z_0^{-1} W_0=\frac{1}{2} (A+I)^{-1}(A-I)$ are of the form $\mu=\frac{1}{2} (\lambda+1)^{-1}(\lambda-1)$, where $\lambda$ is an eigenvalue of $A$. Since $A$ is a singular irreducible M-matrix, then 0 is a simple eigenvalue of $A$, and the remaining eigenvalues have positive real parts. Therefore, the eigenvalues of  $Z_0^{-1} W_0$ are $-1/2$ with multiplicity one, while the other eigenvalues have modulus strictly less than $1/2$.

Now observe that 
$
   \det( \lambda I- \wh Z_0^{-1}\wh W_0)=\det(\wh Z_0^{-1})\det(\lambda \wh Z_0-\wh W_0),
$
and 
\[
\begin{aligned}
\det(\lambda \wh Z_0-\wh W_0)&=\det\Big(\lambda Z_0-W_0-(\lambda+1)Q\Big)\\
&= \det (\lambda Z_0-W_0)\det\big(I-\sigma (\lambda+1)(\lambda Z_0-W_0)^{-1}\bsu \bw^\top \big)\\
&= \det (\lambda Z_0-W_0) \det\big(I-\sigma \frac{(\lambda+1)}{2\lambda+1}\bsu \bw^\top \big)\\
&= \frac{1}{2\lambda+1}\det(\lambda Z_0-W_0)(\lambda(2-\sigma)-(\sigma-1)),
\end{aligned}
\]
where we used the fact that $(\lambda Z_0-W_0)\bsu=(2\lambda+1)\bsu$. 

So that, from the above equalities,
we deduce that the eigenvalues of $\wh Z_0^{-1}\wh W_0$ are those of 
$Z_0^{-1}W_0$, except  for the eigenvalue $-\frac{1}{2}$ which is replaced by $\frac{\sigma-1}{2-\sigma}$.
Let $f(\sigma)=\frac{\sigma-1}{2-\sigma}$, then $f'(\sigma)=\frac{1}{(2-\sigma)^2}>0$, that is, $f(\sigma)\geq f(0)=-\frac{1}{2}$  for $\sigma\geq 0$. It can be verified that $f(\frac{4}{3})=\frac{1}{2}$, so that $|f(\sigma)|<\frac{1}{2}$ for $\sigma\in (0,\frac{4}{3})$. 

The last statement of the theorem follows from \cite[Thm. 5]{bb:pcr}.
\end{proof}

As a consequence of the above result, we deduce that the sequence $\{\wh Z_\ell\}_\ell$ converges quadratically to a limit $\wh S$.
We show that there exists a vector $\bm y$ such that $\wh S=S
+\bsu\by^\top $, where $S=\lim_{\ell\to\infty}  Z_\ell=4A^{1/2}$. We need the following: 

\begin{lemma}\label{lem:12}
    For the sequences generated by \eqref{eq:crshifted}, for $\ell=0,1,\ldots$, we have
    \begin{equation}\label{eq:lem}
    \begin{array}{ll}
    \wh W_\ell\bsu=\omega_\ell\bsu,& \wh Z_\ell\bsu=\zeta_\ell\bsu,\\[1ex]
        \omega_{\ell+1}=-\frac{\omega_\ell^2}{\zeta_\ell},& \omega_0=\sigma-1,\\[1ex]
        \zeta_{\ell+1}=\zeta_\ell+2\omega_{\ell+1},& \zeta_0=2-\sigma.
    \end{array}
    \end{equation}
    Moreover, if $0<\sigma<1$, then $w_\ell<0$, $\zeta_\ell>0$ and $0<-\frac{\omega_\ell}{\zeta_\ell}<\frac{1}{2}$.
\end{lemma}
\begin{proof}
    The proof of \eqref{eq:lem} can be easily carried out by induction and is left to the reader. To prove that $\omega_\ell<0$,  $\zeta_\ell>0$, and $0<-\frac{\omega_\ell}{\zeta_\ell}<\frac{1}{2}$,  we apply induction once again. For $\ell=0$ we have $\omega_0<0$, $\zeta_0>0$ and $0<-\frac{\omega_0}{\zeta_0}<\frac{1}{2}$. For the inductive step, assume that $\omega_\ell<0$, $\zeta_\ell>0$, $0<-\frac{\omega_\ell}{\zeta_\ell}<\frac{1}{2}$  and consider $-\frac{\omega_{\ell+1}}{\zeta_{\ell+1}}$. We have 
    \[
    -\frac{\omega_{\ell+1}}{\zeta_{\ell+1}}=\frac{\omega_\ell^2}{\zeta_\ell} \frac{1}{\zeta_\ell-2\omega_\ell^2/\zeta_\ell}=\frac{\omega_\ell^2/\zeta_\ell^2}{1-2\omega_\ell^2/\zeta_\ell^2}.
    \]
    Since $t:=-\omega_\ell/\zeta_\ell\in(0,1/2)$, then $t^2/(1-2t^2)\in(0,1/2)$ and consequently $\omega_{\ell+1}/\zeta_{\ell+1}<0$. Since $\zeta_\ell>0$, we deduce from \eqref{eq:lem} that $\omega_{\ell+1}$ is negative and from the inequality $\omega_{\ell+1}/\zeta_{\ell+1}<0$ we conclude that $\zeta_{\ell+1}>0$.
\end{proof}

We are ready to prove the following  result:

\begin{theorem}\label{th:mainshift}
For the sequences generated by \eqref{eq:crshifted} we have, for $\ell=0,1,\ldots$,
\begin{equation}\label{eq:wz}
\begin{array}{ll}
\wh Z_\ell=Z_\ell+\bsu\bz_\ell^\top , & \wh W_\ell=W_\ell+\bsu\bw_\ell^\top \\[1ex]
\bw_{\ell+1}^\top =\frac{1}{2} \bw_\ell^\top -(\frac{1}{2} \bz_\ell^\top +\bw_\ell^\top )\wh{Z}_\ell^{-1}\wh{W}_\ell,&  \bw_0=\sigma\bw, \\[1ex]
\bz_{\ell+1}=\bz_\ell+2\bw_{\ell+1},& \bz_0=-\sigma\bw  .
\end{array}
\end{equation}
The limit $\wh S:=\lim_{\ell\to\infty}  \wh{Z}_\ell$ is such that $\wh S-S=\bsu\bz^\top $, where $\bz=\lim_{\ell\to\infty}  \bz_\ell$ and $S:=\lim_{\ell\to\infty}  Z_\ell=4A^{1/2}$, so that $A^{1/2}=4(\wh S-\bsu\bz^\top )$. Moreover,  $S= 16 A\wh S^{-1}$ so that $A^{1/2}=\frac{1}{4}A\wh S^{-1}$. 
Finally, if $\wh{W}_0\le 0$ and $\wh{Z_0}$ is an M-matrix, then $\wh W_\ell\le 0$, and $\wh{Z}_\ell$ are M-matrices, for $\ell=1,2,\ldots$.
\end{theorem}
\begin{proof} To prove \eqref{eq:wz} we proceed by induction. The property is satisfied for $\ell=0$.
For the inductive step, we assume that $\wh Z_\ell=Z_\ell+\bsu\bz_\ell^\top $ and $\wh W_\ell=W_\ell+\bsu\bw_\ell^\top $. Then, from \eqref{eq:crshifted} we find that
\[
\begin{aligned}
     \wh{W}_{\ell+1}-W_{\ell+1}&=-\wh{W}_\ell\wh{Z}_\ell^{-1}\wh{W}_\ell+{W}_\ell{Z}_\ell^{-1}{W}_\ell\\
    & =W_\ell Z_\ell^{-1}(W_\ell-\wh{W}_\ell)+W_\ell Z_\ell^{-1}(\wh{Z}_\ell-Z_\ell)\wh{Z}_\ell^{-1}\wh{W}_\ell\\
    &\phantom{W_\ell Z_\ell^{-1}(W_\ell-\wh{W}_\ell)+W_\ell Z_\ell^{-1}(\wh{Z}_\ell-Z_\ell)Z_\ell^{-1}}
    +(W_\ell-\wh{W}_\ell)\wh{Z}_\ell^{-1}\wh{W}_\ell\\
    &=-W_\ell Z_\ell^{-1}\bsu\bw_\ell^\top +W_\ell Z_\ell^{-1}\bsu\bz_\ell^\top \wh{Z}_\ell^{-1}\wh{W}_\ell-\bsu\bw_\ell^\top \wh{Z}_\ell^{-1}\wh{W}_\ell.\\
\end{aligned}
\]
Since  $W_\ell Z_\ell^{-1}\bsu=-\frac{1}{2}\bsu$, it follows that 
 \[
  \wh{W}_{\ell+1}-W_{\ell+1}=\frac{1}{2} \bsu \bw_\ell^\top -\bsu (\frac{1}{2}\bz_\ell^\top +\bw_\ell^\top )\wh{Z}_\ell^{-1}\wh{W}_\ell,
 \]
so that 
$
\bw_{\ell+1}^\top =\frac{1}{2} \bw_\ell^\top -(\frac{1}{2}\bz_\ell^\top +\bw_\ell^\top )\wh{Z}_\ell^{-1}\wh{W}_\ell
$.
Similarly we do for $Z_{\ell+1}$ and get
\[
\wh Z_{\ell+1}=Z_\ell+\bsu\bz_\ell^\top + 2 W_{\ell+1}+2\bsu \bw_{\ell+1}^\top =
Z_{\ell+1}+\bsu(\bz_\ell+2\bw_{\ell+1})^\top ,
\]
whence $\bz_{\ell+1}=\bz_\ell+2\bw_{\ell+1}$.
Concerning the limit property, since $\wh Z_\ell-Z_\ell=\bsu\bz_\ell^\top $, taking the limit on both sides, we find that
$\bz:=\lim_{\ell\to\infty}  \bz_\ell$ exists and $ \wh S-S=\bsu\bz^\top $.

To prove that $S= 16 A\wh S^{-1}$, we observe that since $A^{1/2}\bsu=0$ and $\wh S\bsu=\bsu(\bz^\top\bsu)$, then $\wh S^2=16A+\bsu\bz^\top \wh S$. Whence we get $\wh S=16A\wh S^{-1}+\bsu\bz^\top $, so that $16A\wh S^{-1}=
\wh S-\bsu\bz^\top =S$. Therefore $A^{1/2}=\frac{1}{4} S=4A\wh S^{-1}$.

Finally, if $\wh{W}_0\le 0$ and $\wh{Z_0}$ is an M-matrix, we proceed by induction on $\ell$ and show that 
\begin{equation}\label{eq:ind}
\wh W_\ell\le0,\quad \wh Z_\ell \hbox{ is an M-matrix}.
\end{equation}
By assumption, if $\ell=0$, equations \eqref{eq:ind} are satisfied. Let us prove the induction step. Assume equations \eqref{eq:ind}  are true for $\ell$ and prove them for $\ell+1$. Since $\wh Z_\ell$ is a 
nonsingular M-matrix then $\wh Z_\ell^{-1}\ge 0$. Since $\wh W_\ell\le 0$ then $\wh W_{\ell+1}=-\wh W_\ell\wh Z_\ell^{-1}\wh W_\ell\le 0$. Since $\wh 
W_{\ell+1}\le 0$ and $\wh Z_\ell$ is an M-matrix, then $\wh Z_{\ell+1}=\wh 
Z_\ell+2\wh W_{\ell+1}$ is a Z-matrix. Let us prove that $\wh Z_{\ell+1}$ is an M-matrix. By Lemma \ref{lem:12} we have $\wh 
Z_{\ell+1}\bsu=\frac{\zeta_\ell^2-2\omega_\ell^2}{\zeta_\ell}\bsu$. On the other hand, by Lemma \ref{lem:12}, $0<\frac{-\omega_\ell}{\zeta_\ell}
<\frac{1}{2}$ so that $\zeta_\ell^2>4\omega_\ell^2$, therefore, $\zeta_\ell^2-
2\omega_\ell^2>0$ and $\wh Z_{\ell+1}\bsu> 0$. Finally, since $\wh Z_{\ell+1}$ is a Z-matrix for which there exists a vector $\bsu>0$ such that $\wh 
Z_{\ell+1}\bsu> 0$ then it is an M-matrix \cite[Theorem 2.3, I27]{bp:book}.
\end{proof}

If $\wh{W}_0\le 0$ and $\wh Z_0$ is an M-matrix, then
from equation \eqref{eq:lem}, we see that the triplet representation for $\wh{Z}_\ell$ is $\wh{Z}_\ell=(-\offdiag(\wh{Z}_\ell), \bsu, \zeta_\ell\bsu)$.
Moreover, if $\sigma=0$, then the sequences $\wh{W}_\ell$ and $\wh{Z}_\ell$ coincide with the sequences $W_\ell$ and $Z_\ell$, generated by CR without shift, equation \eqref{eq:lem} still holds and we have  $Z_\ell\bsu=2^{-\ell+1}\bsu $ and $W_\ell\bsu =-2^{-\ell}\bsu $, we thus have $W_\ell Z_\ell^{-1}\bsu=-\frac{1}{2}\bsu$.

Generally, the matrix $\wh Z_0$ is not necessarily an M-matrix and $\wh{W}_0$ is not necessarily nonpositive, once the shift technique is applied. But if $A$ has a column formed by all nonzero off-diagonal entries, it is possible to choose $\sigma$ and $\bw$ so that $\wh Z_0$ is an M-matrix and $\wh{W}_0\le 0$ so that, by Theorem \ref{th:mainshift}, all the matrices $\wh{Z}_\ell$ are M-matrices and $\wh{W}_\ell\le 0$ for any $\ell\ge 0$. 

To show this property, assume for simplicity that $A=I-C$, $\rho(C)=1$  and that for a given $j$, we have $c_{i,j}\ne 0$, for $i=1,\ldots,n$, $j\ne i$ and $ c_{i,i}\le 1$. 
Let $\gamma=\min_{i,\, i\ne j} c_{i,j}$, then $\gamma\le 1$ and choose $\sigma \le \gamma$. This way, it is easy to show that for $\bw=\frac{1}{u_j}\be_j$, the matrix $C-Q\geq 0$ so that $\wh W_0=-C+Q\le 0$. Moreover, since $\wh Z_0\bsu=(2-\sigma)\bsu$, the condition $\sigma<1$ guarantees that $\wh Z_0$ is an M-matrix. 
Observe that a triplet for the matrix $\wh Z_\ell$ is given by Lemma \ref{lem:12} in the form $(-\offdiag(Z_\ell),\bsu,\zeta_\ell\bsu)$.

\begin{remark}
\rm From the relation $\wh Z_\ell=Z_\ell+\bsu\bz_\ell^\top $, since $Z_\ell\rightarrow S$ linearly while $\wh Z_\ell\rightarrow\wh S$ quadratically, one deduces that $\bz_\ell\rightarrow \bz$ {\em linearly}. Therefore, equations \eqref{eq:wz} cannot be used to compute $A^{1/2}$ in a small number of steps.
    On the other hand, equation $A^{1/2}=\frac{1}{4} A\wh S^{-1}$ is an effective way to compute $A^{1/2}$ through the computation of $\wh S$, which requires a small number of iterations, and its subsequent inversion. However, the problem of implementing the latter formula without encountering numerical cancellation remains unsolved.
\end{remark}

To sum up, we may modify Algorithm \ref{alg:cr} to perform the shifted version of CR by means of triplets. The computation is described in Algorithm \ref{alg:crshifted}, where all the steps are cancellation-free except for the initial step, where $W_0=-C+Q$ is computed, and the last step, where $A^{1/2}$ is computed relying on $A^{1/2}= A\wh Z_{\ell+1}^{-1}$ and cancellation might occur. However, for this step, we can compute the LU factorization of $Z_{\ell+1}$ using the GTH algorithm and then solve two triangular systems.

\begin{algorithm} 
\sml
\textbf{Input:} A triplet representation $(-\offdiag(A),\bsu,0)$ of the singular M-matrix $A$, where $\bsu>0$, $A\bsu= 0$.

\textbf{Output:} The entries of $Z=A^{1/2}$ and its triplet representation $(-\offdiag(Z),\bsu,0)$.

Set $s=\gamma\max_i a_{ii}$, (default: $\gamma=4$)\;
$A\gets A/s$, $\bv\gets \bv/s$\;

Set $C=I-A$\;
Select a column of $C$ formed by all nonzero entries, let it be the $j$th column. If it does not exist, then exit displaying an error message\;

Compute $\gamma = \min_{i}|c_{i,j}|$, set $\sigma = \gamma$, $\bw =\frac{1}{u_j}\be_j$, $Q=\sigma \bsu\bw^\top $\;

Set $W_0=-C+Q$, $Z_0=2(I+A)-Q$, $\omega_0=\sigma-1$, $\zeta_0=2-\sigma$\;

\For{$\ell=0,1,2,\ldots$}{
     given the triplet $(-\offdiag(Z_\ell),\bsu,\zeta_\ell\bsu)$ compute $G_\ell=Z_\ell^{-1}W_\ell$ by means of  Algorithm \ref{alg:gth}\;
     $W_{\ell+1}=-W_\ell G_\ell$\;
     $Z_{\ell+1}=Z_\ell+2W_{\ell+1}$\;
     $\omega_{\ell+1}=-\frac{\omega_\ell^2}{\zeta_\ell}$, $\zeta_{\ell+1}=\zeta_\ell+2\omega_{\ell+1}$\;
     \If{$\|W_{\ell+1}\|\le\varepsilon$}{exit\;}
    }

Given the triplet $(-\offdiag(Z_{\ell+1}),\bsu,\zeta_{\ell+1}\bsu)$ compute $Z=\frac{1}{4} AZ_{\ell+1}^{-1}$ by means of  Algorithm \ref{alg:gth} suitably adjusted to right multiplication for the inverse matrix. 

Set  $Z\gets Zs^{1/2}$.
\caption{\sml Computing $A^{1/2}$ by means of CR, where $A$ is a singular M-matrix}\label{alg:crshifted} 

\end{algorithm}

\subsection{The Incremental Newton iteration}\label{sec:in}
The Incremental Newton (IN)  algorithm is a variation of Cyclic Reduction designed in \cite{Iannazzo2003273}. It is defined as follows
\begin{equation}\label{eq:in}
    \begin{array}{ll}
       X_{\ell+1}=X_\ell+F_\ell,&X_0=A,\\[1ex]
       F_{\ell+1}=-\frac{1}{2}F_\ell X_{\ell+1}^{-1} F_\ell, &F_0=\frac{1}{2}(I-A),
    \end{array}
\end{equation}
where $\lim_{\ell\to\infty}  X_\ell=A^{1/2}$ and $\lim_{\ell\to\infty}  F_\ell=0$.
If $A$ fulfills Assumption \ref{ass:2}, then we can show that $X_\ell$ is an M-matrix for $\ell\ge 0$. We recall the following relations between $X_\ell$ and $F_\ell$ and the matrices $Z_\ell$ and $W_{\ell+1}$ generated by Cyclic Reduction (see \cite{higham:book})
\[
F_{\ell+1}=\frac{1}{2} W_{\ell+1},\quad
X_{\ell+1}=\frac{1}{4} Z_\ell.
\]
Since we know that $Z_\ell$ is an M-matrix and $W_\ell\le 0$ for $\ell\ge 0$, then also $X_\ell$ is an M-matrix and $F_\ell\le 0$ for any $\ell\ge 0$. Moreover, equation \eqref{eq:in} allows us to
provide 
a triplet version of the IN algorithm. 
Replacing the expressions of $W_\ell$ and $Z_\ell$ given by \eqref{eq:in} in \eqref{eq:trip}, we get for $\ell\geq 1$ that 
\begin{equation}\label{eq:in1}
\begin{split}
&\bv_{\ell+1}=X_{\ell+1}\bsu=\frac{1}{4} \bp_{\ell}-F_\ell\bsu, \quad \bp_1=8(I+A)^{-1}\bv,\\
&\bp_{\ell+1}=\bp_\ell-F_\ell X_{\ell+1}^{-1}\bp_\ell.
\end{split}
\end{equation}

The vectors $\bsu$ and $\bv_{\ell+1}$ provide a triplet for the M-matrix $X_{\ell+1}$. Moreover, the above expression for $\bv_{\ell+1}$ can be computed with no cancellation since $F_\ell\le 0$ so that both $\bp_\ell$ and $-4F_\ell\bsu$ are nonnegative.
Observe also that the recursion for $\bp_{\ell+1}$ is obtained through a sum of nonnegative vectors so that it does not generate cancellation.

To implement the algorithm, we have to treat differently the case where $\ell=0$ in fact, \eqref{eq:in1} cannot be applied for $\ell=0$ and must be replaced by $\bp_1=8(I+A)^{-1}\bv$, $X_1\bsu=\frac{1}{2}(\bsu+\bv)$.
In Algorithm \ref{alg:in}, we provide the triplet version of the IN iteration.

\begin{algorithm}
\sml
\textbf{Input:} A triplet representation $(-\offdiag(A),\bsu,\bv)$ of the M-matrix $A$, where $\bsu>0$, $\bv=A\bsu\ge 0$.

\textbf{Output:} The entries of $X=A^{1/2}$  and its triplet representation $(-\offdiag(Z),\bsu,\bw)$.

Set $s=\gamma\max_i a_{ii}$, (default: $\gamma=4$)\;
$A\gets A/s$, $\bv\gets \bv/s$\;

    Set $F_0=\frac{1}{2}(I-A)$, $X_0=A$, $\bp_0=4\bv$, $\bv_0=\bv$, $\widehat {\bv}=\bv$\;
    
    \For {$\ell=0,1,\ldots$}{
       $X_1=X_0+F_0$; 
       
    \If{$\ell=0$} {$\bv_1=\frac{1}{2}(\bsu+\bv)$, $F_1=-\frac{1}{2}F_0X_1^{-1}F_0$, where Algorithm 1 is applied given the triplet $(-\offdiag(X_1),\bsu, \bv_1)$\;}
    
    \If{$\ell=1$} {$\bp_2=8(I+A)^{-1}\bv$, $\bv_2=\frac{1}{4}(\bp_2-4F_1\bsu)$, where $\bp_2$ is computed by means of Algorithm \ref{alg:gth} given the triplet $(-\offdiag(A),\bsu,\bsu+\bv)$\;}
    
    \If {$\ell>1$} {$\bp_{\ell+1}=\bp_\ell-F_{\ell-1}\bg_\ell$, $\bv_{\ell+1}=\frac{1}{4}(\bp_{\ell+1}-4F_\ell\bsu)$\;}

    $[G_\ell,\bg_{\ell+1}]=X_{\ell+1}^{-1}[F_\ell,\bp_{\ell+1}]$,
    where Algorithm \ref{alg:gth} is applied given the triplet $(-\offdiag(X_{\ell+1}),\bsu,\bv_{\ell+1})$\;    
    $F_{\ell+1}=-\frac{1}{2} F_\ell G_\ell$\;
    
    \If {$\|F_{\ell+1}\|\le\varepsilon$} {exit\;}
    }
    $\bd=(\bv_{\ell+1}-\offdiag(X_{\ell+1})\bsu)\oslash\bsu$\;
    
    $X=\hbox{diag}(\bd)+\offdiag(X_{\ell+1})$,\quad $\bw=\bv_{\ell+1}$;

    $\bw\gets \bw s^{1/2}$,\quad $X\gets X s^{1/2}$.

    \caption{\sml Computing $A^{1/2}$ by means of IN}\label{alg:in}
    
\end{algorithm}

Observe that, also for this algorithm, there is potential cancellation in the computation of $F_0$ in the case where $A$ is affected by some error. However, this cancellation can be kept under control as done for the CR, replacing $A$ with $A/s$ with $s>0$ sufficiently large constant (this is needed also to fulfill Assumption~\ref{ass:2}). 

\begin{remark}\label{rem:3}\rm
If the matrix $A$ is singular, say, is the Laplacian of a graph, then there exists $\bsu>0$ such that $A\bsu = 0$, so that $\bv = 0$. Consequently, $\bp_\ell = 0$ and $\bv_\ell = -F_\ell\bsu$ for $\ell=1,2,\ldots$ Moreover, since $F_0\bsu=\frac{1}{2}\bsu$ and $A\bsu = 0$, from \eqref{eq:in} one can inductively prove that $X_\ell\bsu = \frac{1}{2^\ell}\bsu$, $F_\ell\bsu=-\frac{1}{2^{\ell+1}}\bsu$. On the other hand, this fact simplifies the implementation of the algorithm and shows that the convergence of $F_\ell$ to $0$ is only linear.
\end{remark}

As observed in Remark \ref{rem:3}, if $A$ is singular, then the convergence slows to linear. As already done in Section \ref{sec:shift} for the case of CR, we may overcome the slowdown by applying the shift technique.

\section{Error analysis}\label{sec:numan}
In this section, we outline some results useful to perform the error analysis of a single step of the algorithms presented in the previous sections when implemented in floating-point arithmetic with machine precision $\varepsilon$. In the expressions of the error bounds, we consider only the linear part in $\varepsilon$, neglecting the terms containing $\varepsilon^2$. To simplify the notation, we write $a\doteq b$ if $b-a$ is proportional to $\varepsilon^2$; similarly, we write $a\dotle b$ if the coefficient of $\varepsilon$ in the difference $b-a$ is nonnegative.
Moreover, we denote by fl$(\cdot)$ the result of the expression between parentheses obtained in floating point arithmetic.
In our complexity bounds, we consider only the asymptotically dominant expression. For instance, in the bound $\psi(n)=2n^3+3n^2+n+1$ we consider only the asymptotic dominant term $2n^3$, and we write $\psi(n)\cong 2n^3$.

We recall that if $\widehat a=a(1+\delta)$ then $\delta=(\wh a-a)/a$ is the relative error in the approximation of $a$ by $\wh a$. We also recall that
$\fl(a*b)=(a*b)(1+\delta)$, $|\delta|\le \varepsilon$
where $*$ is any of the four arithmetic operations $+,-,\times,/$.

Given a positive integer $n$, denote $M(n)=\lceil\log_2 n\rceil$, and consider the following recursive algorithm to compute $w(a_1,\ldots,a_n)=\sum_{i=1}^n a_i$.

\begin{algorithm}\label{alg:sum}\rm
\sml
{\bf Input}: $a_i\in\mathbb R$, $i=1,2,\ldots,n$.

{\bf Output}: $w(a_1,\ldots,a_n)=\sum_{i=1}^n a_i$

\begin{enumerate}
    \item if $n=1$, output $w(a_1)=a_1$;
    \item if $n=2$ output $w(a_1,a_2)=a_1+a_2$;
    \item if $n>2$  compute $s_1=w(a_1,\ldots,a_m)$, $s_2=w(a_{m+1},\ldots,a_n)$, where $m=\lceil \frac {n}{2}\rceil$;
    \item output $s_1+s_2$.
\end{enumerate}
\caption{\sml Binary tree sum}
\end{algorithm}

We have the following simple result. For the sake of completeness, we report the standard proof.
\begin{lemma}
Let $\wh w=\fl(\sum_{i=1}^n a_i)$ be the  result obtained by applying Algorithm~\ref{alg:sum} in floating point arithmetic.   Then $\wh w= \sum_{i=1}^n a_i(1+\delta_i^{(n)})$, $|\delta_i^{(n)}|\dotle M(n)\varepsilon$. Moreover, if $a_i\ge 0$, $i=1,\ldots,n$, then the relative forward error is bounded by $|(\wh w-w)/w|\dotle M(n)\varepsilon$.
Similarly, we have $\fl(\sum_{i=1}^na_ib_i)=\sum_{i=1}^n a_ib_i(1+\theta_i^{(n)})$, with $|\theta_i^{(n)}|\dotle (M(n)+1)\varepsilon$, and if $a_ib_i\ge 0$, $i=1,\ldots,n$, then also the forward error is bounded from above  by $(M(n)+1)\varepsilon$.
\end{lemma}
\begin{proof} To prove the expression of $\wh w$, 
we proceed by induction. For $n=2$ we have $\widehat w=\fl(a_1+a_2)=(a_1+a_2)(1+\delta)$, $|\delta|\le\varepsilon=M(2)\varepsilon$.
For the inductive step, we have $\wh w=(\wh s_1+\wh s_2)(1+\delta)$, $|\delta|\le\varepsilon$.
By the inductive assumption, we have $\wh s_1=\sum_{i=1}^ma_i(1+\delta_i^{(m)})$, $\wh s_2=\sum_{i=m+1}^{n}a_i(1+\delta_i^{(n-m)})$, with $|\delta_i^{(m)}|\dotle M(m)\varepsilon$, $|\delta_i^{(n-m)}|\dotle M(n-m)\varepsilon$, so that $\wh w= \sum_{i=1}^n a_i(1+\delta_i^{(n)})$, where $\delta_i^{(n)}\doteq \delta_i^{(m)}+\delta$ for $i=1,\ldots,m$ and $\delta_i^{(n)}\doteq \delta_i^{(n-m)}+\delta$ for $i=m+1,\ldots,n$. The proof is completed since $\log_2 m\le \log_2\frac {n}{2}=\log_2n-1$, so that $M(m)=\lceil\log_2 m\rceil\le \lceil\log_2 n\rceil-1=M(n)-1$ and
the same bound holds for $M(n-m)$ since $n-m\le m$.
Concerning the forward error, if $a_i\ge 0$, we have the upper bound $|\sum_{i=1}^n\delta_i^{(n)}a_i|/\sum_{j=1}^n a_j\le \max_i |\delta_i^{(n)}|\cdot \sum_{i=1}^n a_i/\sum_{j=1}^n a_j\dotle M(n)\varepsilon$. 
For the scalar product $\sum_{i=1}^n a_ib_i$, we proceed similarly.
\end{proof}

In \cite{alfa} it is proved that if $\wh x$ is the floating point vector obtained by the execution of Algorithm \ref{alg:gth} applied to the system $Ax=b$, then $|\wh x-x|\dotle\varphi(n) |x|\varepsilon$, where $\varphi(n)\cong \frac{4}{3} n^3$ and the summations are computed with the customary algorithm. On the other hand, if summations are computed using Algorithm \ref{alg:sum}, then it can be proved that the error bound holds with $\varphi(n)\cong \frac{4}{3} n^2\lceil\log_2 n\rceil$.

The main operations performed by the algorithms described in the previous sections are essentially two, namely, the solution of $n\times n$ linear systems with an M-matrix, and the multiplication of two nonnegative matrices. The relative error in the former operation is amplified by the factor $\varphi(n)\cong\frac{4}{3} n^2\lceil\log_2 n\rceil$. The latter operation generates a relative error that is bounded by $(M(n)+1)\varepsilon$ if the sum in the  row-by-column products is computed using Algorithm \ref{alg:sum}.
Therefore, the overall relative error in one step of the algorithms based on triplets is of the order of $(n^2\log_2 n)\varepsilon$.

According to the analysis of \cite{poloniCR}, this worst-case error bound would imply a relative error in the matrices generated by CR of the order $2^\ell\varphi(n)$, that is, an exponential growth as a function of the step $\ell$. From the numerical experiments that we present in the next section, it turns out that this exponential growth of the error is never encountered.

\section{Numerical experiments}\label{sec:numexp}

In this section, we present the results of some numerical experiments that we have performed in MATLAB, version
9.11.0.2358333, 
on a platform using the Linux Ubuntu operating system, with Intel i3-7100 CPU \@ 3.90GHz $\times$ 4, and 16 GB RAM.

The matrices in our tests are singular irreducible M-matrices and nonsingular M-matrices. In all the tests, we computed the principal matrix square root in high-precision arithmetic relying on the toolbox Advanpix and then in standard floating point precision using the different available algorithms. In all the numerical experiments, the component-wise relative errors have been computed using the command 
{\tt E = abs((S-Smp)./Smp);} where {\tt Smp} is the matrix square root computed in high precision with the command {\tt sqrtm}, while {\tt S} is the matrix square root computed in standard precision with the tested algorithms.

In the tables, we denote Algorithms \ref{alg:cr}, \ref{alg:crshifted}, and \ref{alg:in} by {\tt CR}, {\tt CR-sing}, and {\tt IN} respectively. Moreover, {\tt sqrtm} denotes the MATLAB command for the matrix square root, while {\tt CR-stnd} is the standard implementation of Cyclic Reduction where the MATLAB command backslash is used to compute expressions of the type $M^{-1}N$.

\subsection*{Test 1: Graph Laplacian} As a first test, we consider the $n\times n$ Laplacian $A=D-C$ of a directed network, where $C$ is the adjacency matrix and $D=\operatorname{diag}(\bd)$ is the diagonal matrix with the vector $\bd=C\ones$ on the diagonal, so that $A\bsu = 0$ for $\bsu = \ones$. In this example $C=(c_{ij})$ is the companion matrix  associated with the polynomial $p(x)=x^n-\sum_{i=0}^{n-1}x^i$, that is,  
$c_{ij}=1$ if $j=i-1$, $c_{in}=1$, while $c_{ij}=0$, otherwise.

In Table \ref{tab:frob_err}, we report the maximum relative component-wise error obtained in for several values of $n$ for the different algorithms we have analyzed. In Figure \ref{fig:sol} we display, in log scale, the absolute value of the solution computed by the CR algorithm based on triplets, and the one computed by the MATLAB command {\tt sqrtm}, while in Figure \ref{fig:err} we display the log plot of the relative errors of the different algorithms. 

\begin{table}
    \centering \footnotesize
    \begin{tabular}{c|ccc|cc}
$n$ &  CR      & IN       & CR-sing  & CR-stnd  &{\tt sqrtm}\\ \hline
10  & 7.7e-16  & 7.7e-16  & 5.2e-15  & 9.2e-09  & 2.3e-08   \\ 
20  & 1.3e-15  & 1.3e-15  & 1.3e-14  & 1.9e-08  & 6.5e-08   \\ 
50  & 2.9e-15  & 2.9e-15  & 5.5e-14  & 2.5e-08  & 1.3e+01   \\ 
100 & 1.8e-15  & 1.8e-15  & 1.2e-13  & 4.1e-08  & 1.2e+16   \\
200 & 6.3e-15  & 6.3e-15  & 2.1e-13  & 3.5e-08  & 1.3e+46   \\ 
400 & 9.8e-15  & 9.83e-15 & 4.4e-13  & 2.0e-08  & 1.3e+107  \\  
    \end{tabular}
    \caption{\rm Component-wise relative errors for Test 1.}
    \label{tab:frob_err}
\end{table}

We may notice that, since the solution has both large and small entries in modulus, the solution computed by the MATLAB command {\tt sqrtm} is affected by very large relative errors in the smallest modulus entries. In contrast, the solutions computed by the algorithms based on the triplet representation provide highly accurate results. Standard CR, implemented without using triplets, produces larger relative errors of the order of the square root of the machine precision. Figure \ref{fig:err} gives an idea of the error distribution. Observe also that Cyclic Reduction and Incremental Newton perform the same, and quite well, since they are slightly different implementations of the same algorithm.

\begin{table}
    \centering \footnotesize
    \begin{tabular}{c|ccc|cc}
$n$ &  CR      & IN       & CR-sing  & CR-stnd  &{\tt sqrtm}\\ \hline
10  &  1.1e-02 &  2.3e-02 &  3.7e-03 &  1.2e-03 &  6.7e-04 \\ 
20  &  1.6e-02 &  3.8e-02 &  1.4e-02 &  3.3e-03 &  7.2e-03 \\ 
50  &  4.9e-02 &  4.2e-02 & 5.5e-14  &  9.9e-03 &  5.5e-03 \\ 
100 &  1.1e-01 &  1.6e-01 &  2.8e-02 &  1.8e-02 &  1.4e-02 \\
200 &  5.6e-01 &  7.2e-01 &  1.0e-01 &  3.0e-02 &  5.7e-02 \\ 
400 &  4.7e+00 &  4.9e+00 &  8.6e-01 &  2.3e-01 &  1.8e-01\\  
    \end{tabular}
    \caption{\rm CPU time, in seconds, for Test 1.}
    \label{tab:frob_cpu}
\end{table}

\begin{figure}
\centering
\includegraphics[scale=0.4]{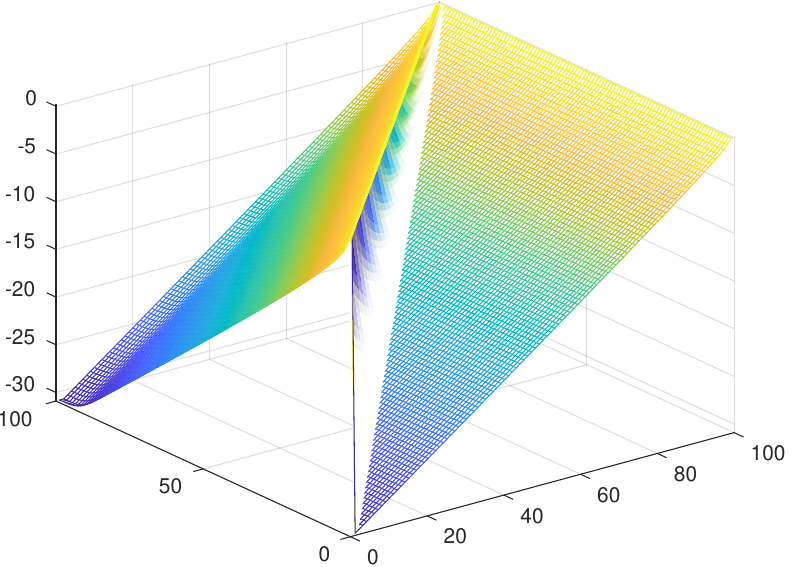}
\includegraphics[scale=0.4]{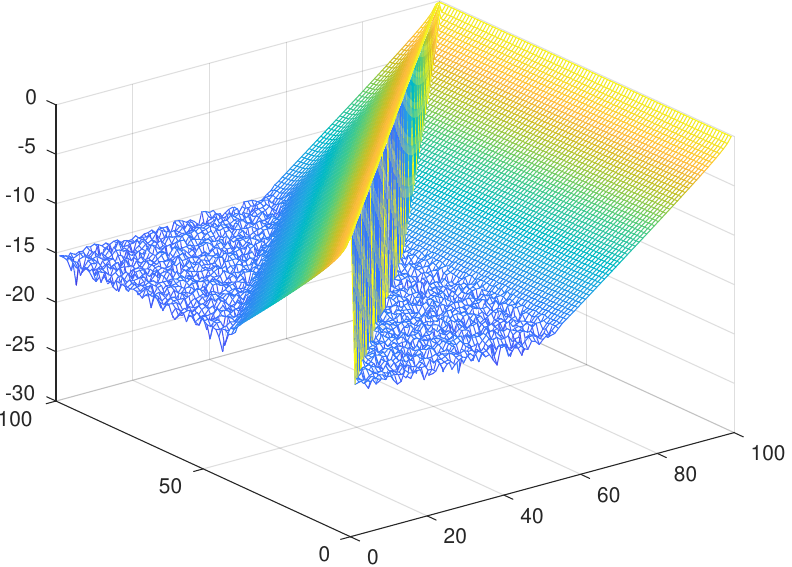}
\caption{Log plot of the absolute value of the matrix square root in Test 1 with $n=100$, computed by means of 
Algorithm \ref{alg:cr} (left), and by the MATLAB command {\tt sqrtm} (right).}\label{fig:sol}
\end{figure}

\begin{figure}
\centering
\includegraphics[scale=0.4]{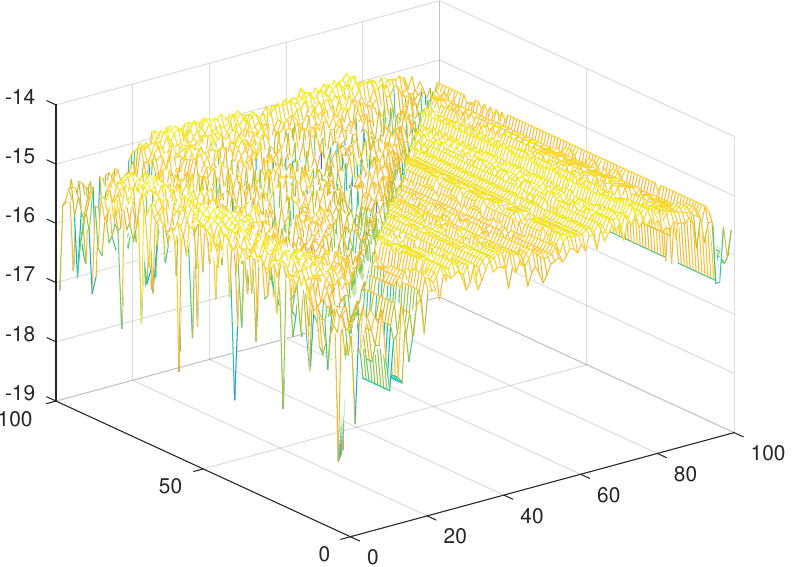}
\includegraphics[scale=0.4]{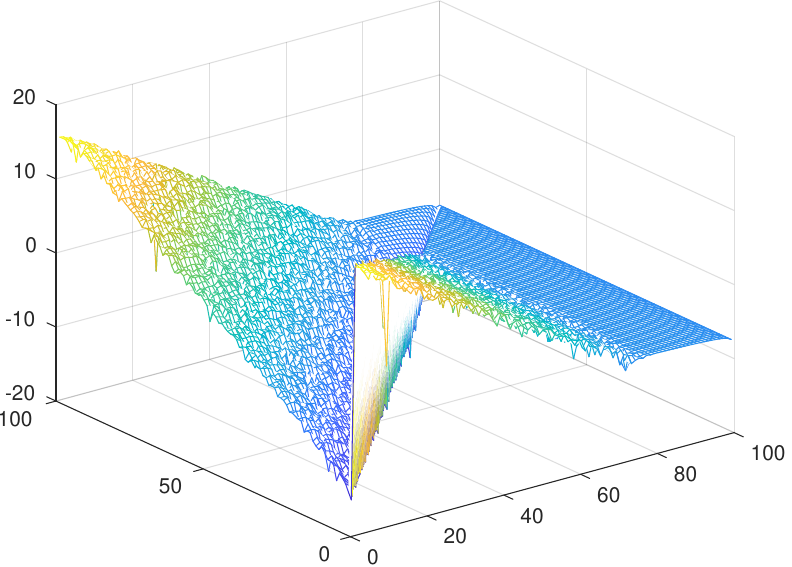}\\\includegraphics[scale=0.4]{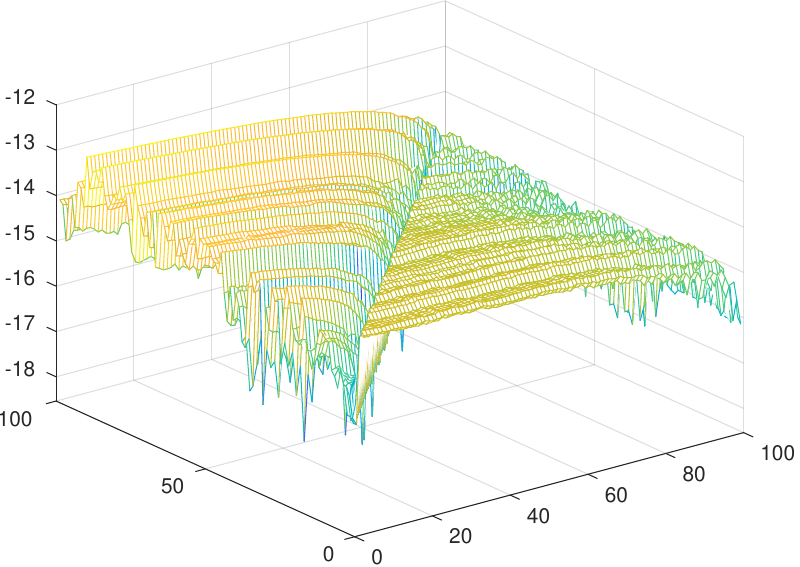}
\includegraphics[scale=0.4]{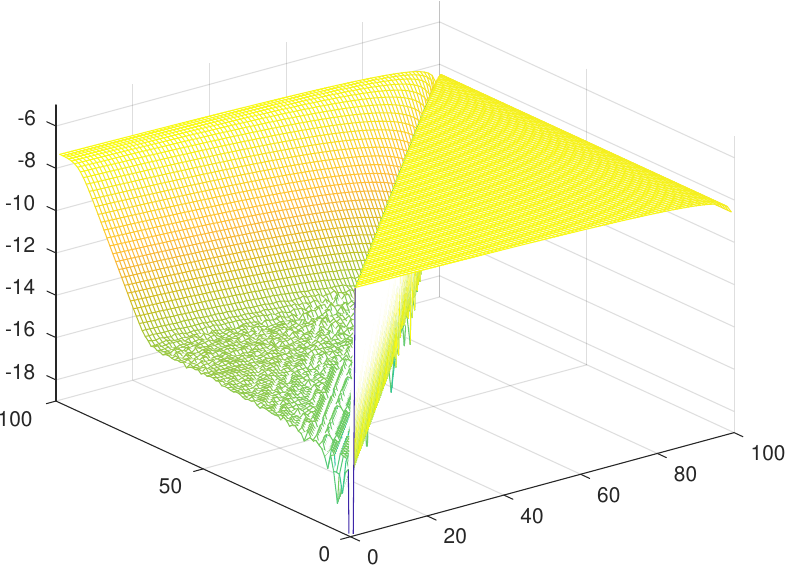}
\caption{Log plot of the relative errors in the matrix square root of Test 1 with $n=100$, computed by means of Algorithm \ref{alg:cr} (top left), by the Matlab command {\tt sqrtm} (top right),  by Algorithm \ref{alg:crshifted} (bottom left), and by standard CR (bottom right).}\label{fig:err}
\end{figure}

The  CPU time needed by these algorithms is reported in Table \ref{tab:frob_cpu}. We may observe that {\tt sqrtm} has the shortest time in almost all the runs. This is mainly because this command is built in the MATLAB package and relies on the LAPACK subroutines. The CPU time needed by our implementation is mainly taken by Algorithm \ref{alg:gth}, more precisely by the update of the off-diagonal entries of the matrix $U$.
Indeed, this computation does not rely on the LAPACK subroutines. 
Observe that, as expected, the accelerated version of CR given in Algorithm \ref{alg:crshifted} is faster than CR based on triplets implemented in Algorithm \ref{alg:cr}. 

However, this algorithm may suffer numerical instability in cases where the matrix $Z$ inverted at the end of the procedure is very ill-conditioned. This happens typically when the vector $\bsu $ in the kernel of $A$ has very unbalanced components, i.e., when the ratio $\varepsilon= \min u_i/\max u_i$ is small. This is shown in the next Test 2.

\subsection*{Test 2: Unbalanced singular M-matrix} The second test concerns a singular irreducible M-matrix $A$ such that $A\bsu=0$, $\bsu>0$, $\min u_i / \max u_i = \varepsilon$, for different values of $\varepsilon<1$.
Observe that $\bsu$ has unbalanced components when $\varepsilon$ takes small values. In this case, we expect that the matrix $Z$ generated by Algorithm \ref{alg:crshifted}, which must be inverted in the last step, is ill-conditioned, the smaller $\varepsilon$. The matrices of this test are obtained with $A=D-B$, where $B$ is a given nonnegative matrix, $D=\mbox{diag}(\bd)$, $\bd=(d_i)$, $d_i=w_i/u_i$, $\bw=B\bsu$. In our experiments, we have chosen $u_1=\varepsilon$, $u_i=1$ for $i>1$, and $B=(b_{ij})$, the upper Hessenberg matrix with unit entries, i.e., such that $b_{ij}=1$ for $i\le j+1$, $b_{ij}=0$ elsewhere.

We have chosen values of the unbalancing ratio $\varepsilon$ of the components of $\bsu$ in the set $\{10^{-2}, 10^{-5}, 10^{-8}, 10^{-11}, 10^{-14}\}$, and we have set $n=100$. Table \ref{tab:tst2} reports the values of the maximum relative error, of the different algorithms, together with the values of cond$(Z)$.

\begin{table} 
	\centering \footnotesize
	
	\begin{tabular}{c|ccc|cc|c}
		$\varepsilon$ &  CR      & IN       & CR-sing  & CR-stnd  &{\tt sqrtm}& cond$(Z)$\\ \hline
		$10^{-2}$  &  2.0e-14 &  2.0e-14 &  8.0e-12 &  1.6e-08 &  2.1e+110& 8.8e+02\\ 
		$10^{-5}$  &  1.0e-14 &  1.0e-14 &  6.7e-09 &  5.1e-06 &  1.1e+58& 2.8e+04 \\ 
		$10^{-8}$  &  2.3e-13 &  2.3e-13 &  4.39e-06 & 3.7e-02 &  4.3e+58&8.8e+05 \\ 
		$10^{-11}$ &  1.0e-13 &  1.0e-13 &  6.6e-03 &  2.8e+03 &  3.1e+139& 2.8e+07 \\
		$10^{-14}$ &  3.2e-14 &  3.2e-14 &  4.3e+03 &  4.2e+07 &  3.7e+47& 2.9e+16 \\  
	\end{tabular}
	\caption{\rm Component-wise relative errors for Test 2, together with the condition number of the matrix $Z$ in Algorithm \ref{alg:crshifted}, for different values of the unbalancing ratio $\varepsilon$ and with $n=100$.}
	\label{tab:tst2}
\end{table}

We may see that while the algorithms CR and IN provide high precision solutions independently of the unbalancing factor $\varepsilon$, the maximum component-wise relative error of Algorithm \ref{alg:crshifted} grows as $\varepsilon$ decreases. This behavior was expected since in the last step of the algorithm, the solution is obtained by multiplying the matrix $A$ by $Z^{-1}$, and the latter matrix has a condition number that grows as $\varepsilon$ decreases.

Observe also that the maximum component-wise relative error in the solution obtained with the MATLAB command {\tt sqrtm} reaches very large values. On one hand, this is because $A^{1/2}$ has some entries which are much below the machine precision. On the other hand, in this case, the command {\tt sqrtm} seems to fail in computing the smallest lower triangular entries of $A^{1/2}$.

\begin{figure}
	\centering
	\includegraphics[scale=0.4]{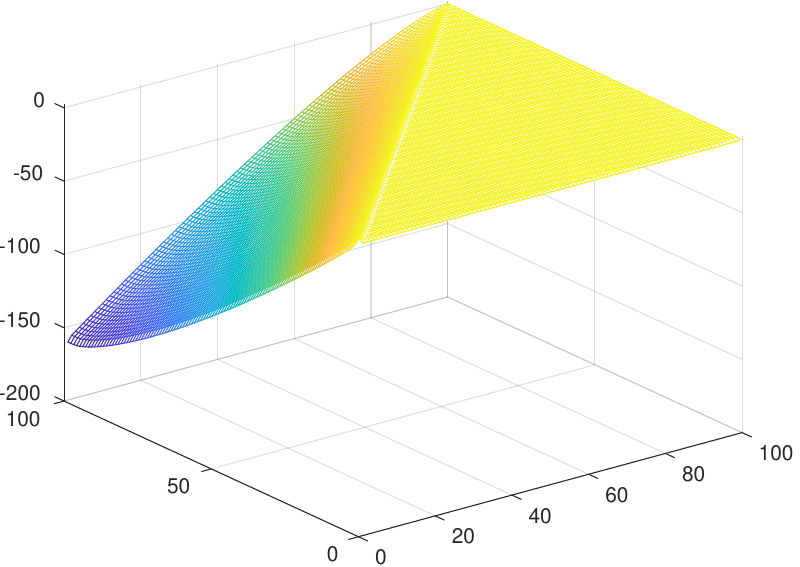}
	\includegraphics[scale=0.4]{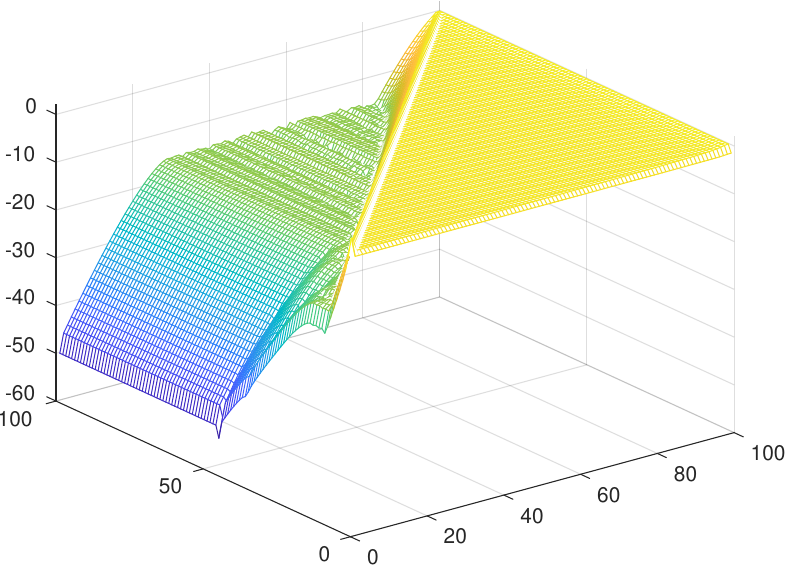}
	\caption{Log plot of the absolute value of the matrix square root in Test 2 with $n=100$, and $\varepsilon=10^{-8}$, computed by means of Algorithm \ref{alg:cr} (left), and by the MATLAB command {\tt sqrtm} (right).}\label{fig:sol1}
\end{figure}

\begin{figure}
	\centering
	\includegraphics[scale=0.4]{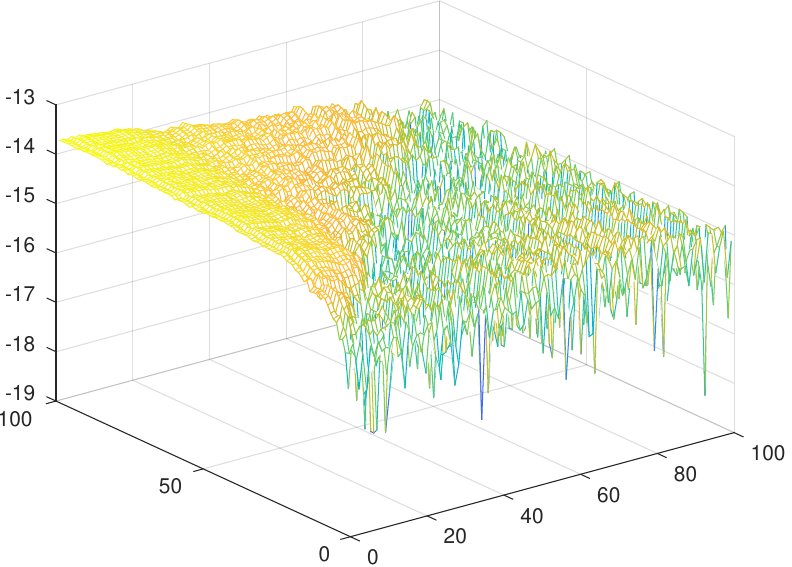}
	\includegraphics[scale=0.4]{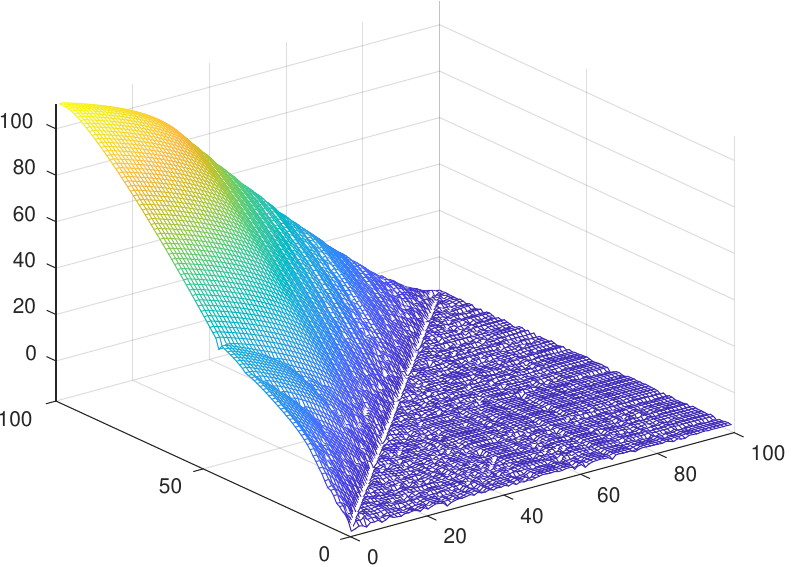}\\\includegraphics[scale=0.4]{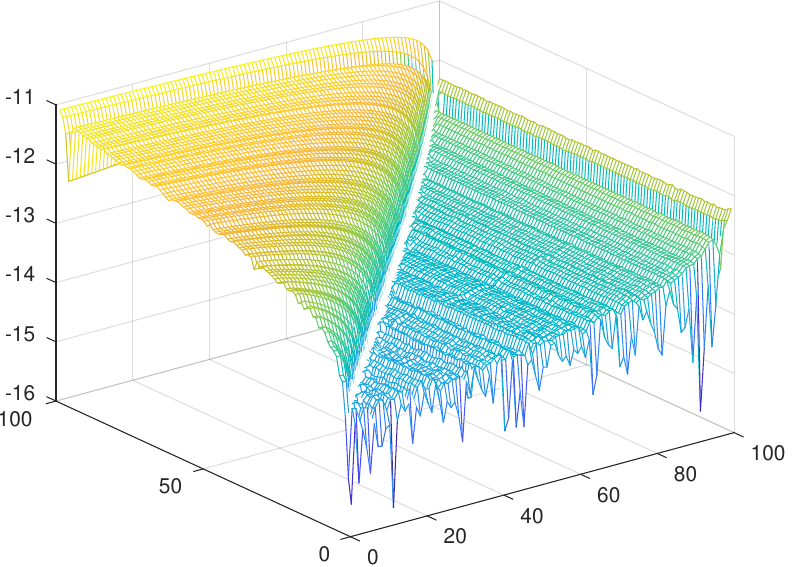}
	\includegraphics[scale=0.4]{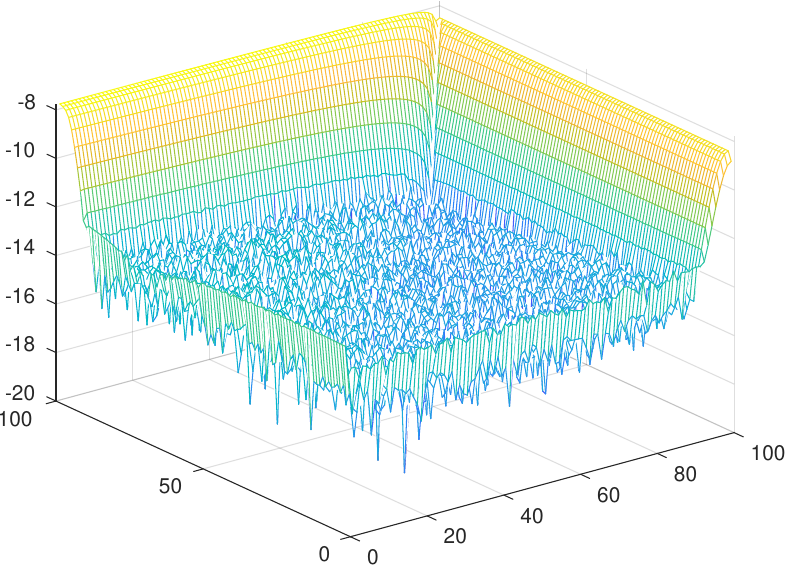}
	\caption{Log plot of the relative errors in the matrix square root of Test 3 with $n=100$, and $\varepsilon=10^{-2}$, computed by means of Algorithm \ref{alg:cr} (top left), by the MATLAB command {\tt sqrtm} (top right),  by Algorithm \ref{alg:crshifted} (bottom left), and by standard CR (bottom right).}\label{fig:err1}
\end{figure}

Figure \ref{fig:sol1} displays, in log scale, the modulus of the entries of $A^{1/2}$ computed with CR, and  with the command {\tt sqrtm} for $\varepsilon=10^{-8}$. Figure \ref{fig:err1} displays the relative errors generated by the different algorithms.

\subsection*{Test 3: Nonsingular M-matrix}
Now, we consider the case where $A$ is a nonsingular and well-conditioned matrix. More precisely,  we set $A=(a_{ij})$, $a_{ij}=-1$ for $j-i>0$, and for $0<i-j<n/4$, while $a_{ii}=n$, and $a_{ij}=0$ elsewhere. In this case, $A$ is nonsingular, $\hbox{cond}_2(A)< 4$. From the shape of $A^{1/2}$  shown in Figure \ref{fig:sol3}, we may observe that the moduli of the entries are in the range $[10^{-6},10^2]$.
Table \ref{tab:tst3} reports the component-wise relative errors of the different algorithms.
\begin{figure}
	\centering
	\includegraphics[scale=0.4]{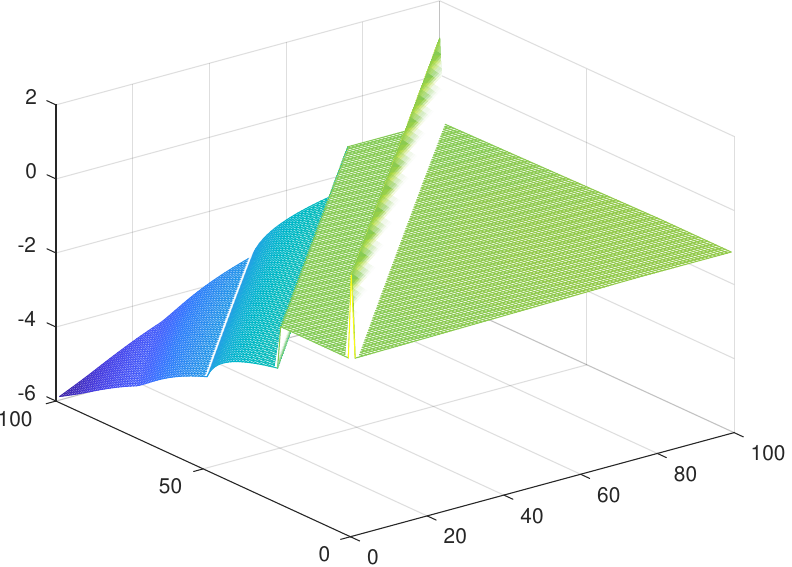}
\caption{Log plot of the absolute value of the matrix square root in Test 3 with $n=100$.}\label{fig:sol3}
\end{figure}

In this test, we observe that not only do the versions of CR and IN based on triplet representation perform accurately, but standard CR also produces very accurate results. At the same time, {\tt sqrtm} fails to provide accurate approximations of the entries of smallest modulus.

We may also observe that, since CR and IN compute the solution in a few iterations, the CPU time taken by these algorithms is close to that taken by {\tt sqrtm}. For instance, for $n=400$,  the execution of {\tt sqrtm} requires {\tt 1.4e-01} seconds, while CR, IN, and standard CR  require {\tt 7.0e-01}, {\tt 6.8e-01}, and {\tt 6.56e-02} seconds, respectively.

\begin{table} 
	\centering \footnotesize
	\begin{tabular}{c|ccc|cc}
		$n$ &  CR      & IN       &  CR-stnd  &{\tt sqrtm}\\ \hline
		$10$&  8.7e-16 &  1.2e-15 &  4.0e-15 &  8.7e-06\\ 
		$20$&  1.2e-15 &  1.4e-15 &  4.4e-15 &  4.4e-04 \\ 
		$50$&  1.8e-15 &  1.8e-15 &  1.7e-14 &  3.3e-09 \\ 
    	$100$& 4.0e-14 &  1.7e-13 &  6.8e-14 &  3.7e-09 \\
		$200$&  3.4e-15 & 1.7e-14 &  3.8e-13 &  1.1e-08\\  
  		$400$&  2.7e-13 & 1.4e-12 &  1.7e-12 &  2.5e-08\\  
	\end{tabular}
	\caption{\rm Component-wise relative errors for Test 3 for different values of the matrix size $n$.}
	\label{tab:tst3}
\end{table}

\section{Conclusions}\label{sec:conc}

The triplet representation of an M-matrix $A$, that exists under customary assumptions (e.g., $A$ invertible or singular irreducible), is a useful tool to design algorithms whose accuracy is not related to the conditioning, because they are mostly based on sums of numbers of the same sign.

We have applied this technology to the computation of the square root $A^{1/2}$ in an accurate way, for the first time, as far as we know, adapting some of the existing algorithms, i.e., the CR and the IN iterations. The numerical results are striking.

Some algorithms for the square root, such as the Bj\"orck-Hammarling method, destroy from the beginning the positivity structure; while others, like the Newton method and the Denman-Beavers method, produce sequences of M-matrix, but we were not able to describe them in terms of triplets. 

The CR, besides being related to the Newton method, is also a method to solve a quadratic matrix equation, which in our article is obtained by implicitly applying the Cayley transform to the equation $X^2-A=0$. In future work, we plan to study other methods for the square root based on quadratic matrix equation solution with CR and Doubling algorithms.

A related problem that deserves future research is the computation of the $p$th root of an M-matrix $A$, which shares with the square root the fact that the Newton method applied to $X^p-A=0$ produces a sequence of M-matrices.

\section*{Acknowledgments}
The authors wish to thank Fabio Durastante and Federico Poloni for fruitful discussions and for providing some references and applications.

This work has been partially supported by: INdAM through a GNCS project; PRA 2022 WP 4.1 ``Rationalists'' funded by the University of Perugia; the Italian Ministry of University and Research (MUR) through the PRIN 2022ZKME7 CUP B53C24006410006; CIAICO/2023/275 ``Optimized computation of matrix functions and applications to artificial intelligence'';  the Italian Ministry of University and Research (MUR) through the PRIN 2022 ``Low-rank Structures and Numerical Methods in Matrix and Tensor Computations and their Application'' code 20227PCCKZ MUR D.D. financing decree n. 104 of February 2nd, 2022 (CUP I53D23002280006), and through the MUR Excellence Department Project awarded to the Department of Mathematics, University of Pisa, CUP I57G22000700001.
The second and the third author are also affiliated to the INdAM-GNCS (Gruppo Nazionale di Calcolo Scientifico). The fourth author is supported  by the National Natural Science Foundation of China under grant No. 12201591.

\end{document}